\documentclass[11pt,a4paper,reqno,backref=page]{amsart}

\usepackage{amsmath,amssymb,amsthm}
\usepackage{mathrsfs,eucal}
\usepackage{hyperref}
\hypersetup{%
  hidelinks,
  pdftitle={Lp bounds for wave operators with critical
    electromagnetic potentials},
  pdfauthor={Piero D'Ancona, Xitao Gao, and Junyong Zhang},
  pdfsubject={Electromagnetic scattering in the plane and
    sharp weighted Aharonov--Bohm bounds},
  pdfkeywords={wave operators, critical electromagnetic
    potentials, Aharonov--Bohm Hamiltonian, power weights,
    Friedrichs extension, Krein extension,
    spectral transference}
}

\newcommand{\R}{\mathbb R}
\newcommand{\Z}{\mathbb Z}
\newcommand{\C}{\mathbb C}
\newcommand{\T}{\mathbb T}
\newcommand{\N}{\mathbb N}
\newcommand{\HH}{\mathcal H}                 % Hankel transform
\newcommand{\Mell}{\mathcal M}               % Mellin transform
\newcommand{\ind}{\mathbf 1}
\newcommand{\dd}{\,d}
\DeclareMathOperator*{\slim}{s\text{-}lim}
\DeclareMathOperator{\Res}{Res}
\let\Re\relax\DeclareMathOperator{\Re}{Re}
\newcommand{\Gam}[2]{\frac{\Gamma(#1)}{\Gamma(#2)}}
\newcommand{\coloneq}{:=}

\usepackage{amsmath,amsfonts,verbatim}
\usepackage{latexsym}
\usepackage{amssymb,leftidx}
\usepackage{extarrows}
\usepackage{overpic}
\usepackage{color}
\usepackage{epsfig}
\usepackage{subfigure}
\usepackage{tikz}

\numberwithin{equation}{section}
\newtheorem{theorem}{Theorem}[section]
\newtheorem{lemma}[theorem]{Lemma}
\newtheorem{proposition}[theorem]{Proposition}
\newtheorem{corollary}[theorem]{Corollary}
\theoremstyle{definition}

\theoremstyle{remark}
\newtheorem{remark}[theorem]{Remark}

\begin{document}

\title[$L^p$ bounds for electromagnetic wave operators]{
  $L^p$ bounds for wave operators with
  critical electromagnetic potentials}

\author{Piero D'Ancona}
\address{Piero D'Ancona:
Department of Mathematics ``Guido Castelnuovo'', Sapienza
University of Rome, Piazzale Aldo Moro 5, 00185 Rome, Italy}
\email{dancona@mat.uniroma1.it}

\author[Xitao Gao]{Xitao Gao}
\address{Xitao Gao\newline
Department of Mathematics, Beijing Institute of Technology,
Beijing 100081, P.\,R.\ China}
\email{xitao\_gao@bit.edu.cn}

\author[Junyong Zhang]{Junyong Zhang}
\address{Junyong Zhang\newline
Department of Mathematics, Beijing Institute of Technology,
Beijing 100081, P.\,R.\ China}
\email{zhang\_junyong@bit.edu.cn}

\subjclass[2020]{35Q40, 42B15, 42B25, 47A40, 81U05}
\keywords{Wave operators, critical electromagnetic potentials,
  Aharonov--Bohm Hamiltonian, power weights,
  Friedrichs and Krein extensions, spectral transference}

\thanks{The first author is supported by the Sapienza
  University research projects 
  ``Wave dynamics in heterogeneous media''
  (DYNAWAVES), CUP B83C25000880005, and ``Bridging Analysis and
  Computation in Evolutionary PDEs'', 
  and by the Gruppo Nazionale per
  l'Analisi Matematica, la Probabilit\`a e le loro 
  Applicazioni (GNAMPA)
  of the Istituto Nazionale di Alta Matematica (INdAM).}

\begin{abstract}
  We study the M\o ller wave operators for scaling critical
  electromagnetic Hamiltonians in the plane. For smooth
  transverse magnetic and angular electric potentials, with
  nonnegative angular operator and magnetic flux outside
  $\frac12\Z$, we prove that the wave operators relative to
  $-\Delta$ exist, are unitary on $L^2$, and, together with
  their adjoints, are bounded on every $L^p$, $1<p<\infty$.

  We then specialize to the the Aharonov--Bohm model
  and we determine the exact ranges for the boundedness
  of their wave operators on weighted 
  $L^p(\R^2,|x|^\beta\,dx)$ spaces,
  for both the Friedrichs and the Krein realizations.
  In the Friedrichs case, this gives in particular 
  the already known
  boundedness on all $L^{p}$ spaces $1<p<\infty$,
  while boundedness fails at $p=1,\infty$.
  In the Krein case, both wave
  operators and adjoints are bounded
  precisely when $2/(2-\eta_\alpha)<p<2/\eta_\alpha$, where
  $\eta_\alpha=\max\{\alpha,1-\alpha\}$
  (here $\alpha\in(0,1)$). Thus the boundary
  condition changes the admissible exponents.
\end{abstract}

\maketitle

\section{Introduction and main results}\label{sec:intro}

We consider the scaling critical electromagnetic Hamiltonian
on $\R^2$
\begin{equation}\label{eq:general-H}
  \mathcal L_{\mathbf A,a}
  =\left(-i\nabla+\frac{\mathbf A(\theta)}r\right)^2
  +\frac{a(\theta)}{r^2},\qquad
  x=r(\cos\theta,\sin\theta).
\end{equation}
We assume a transverse magnetic potential:
$\mathbf A(\theta)=b(\theta)e_\theta$, where
$e_\theta=(-\sin\theta,\cos\theta)$, and $a,b$ are real.
Writing $D_\theta=-i\partial_\theta$, define the angular
operator and the magnetic flux by
\begin{equation}\label{eq:general-angular}
  B_{\mathbf A,a}=(D_\theta+b(\theta))^2+a(\theta),
  \qquad
  \alpha=\frac1{2\pi}\int_0^{2\pi}b(\theta)\,d\theta.
\end{equation}
When $B_{\mathbf A,a}\ge0$, we take the Friedrichs
realization of \eqref{eq:general-H}, obtained by closing its
nonnegative quadratic form on
$C_c^\infty(\R^2\setminus\{0\})$.

Our first goal is to prove boundedness on $L^p$ for the
scattering wave operators relative to the free
Laplacian. Our second goal is to identify sharp parameter
ranges for weighted $L^{p}$ estimates,
and endpoint obstructions, in the Aharonov--Bohm model,
and to determine how the choice between the Friedrichs and
Krein realizations changes those ranges.
Recall that for a selfadjoint Hamiltonian $H$, 
the wave operators are defined as
\begin{equation}\label{eq:general-wave}
  W_\pm(H,-\Delta)
  =\slim_{t\to\pm\infty}e^{itH}e^{it\Delta},
\end{equation}
when this strong limit exists on $L^2(\R^2)$. More generally,
$W_\pm(H_1,H_0)$ denotes the limit of
$e^{itH_1}e^{-itH_0}$ with the same sign convention.

The usefulness of $L^p$ wave operator bounds comes from the
intertwining identity
$m(H)P_{\rm ac}(H)=W_\pm m(-\Delta)W_\pm^*$.
Bounds on the two factors transfer free spectral estimates
to the perturbed operator. For short range Schr\"odinger
perturbations this program was developed by Yajima
\cite{Yajima1995,Yajima1999}; see also
\cite{JensenYajima2002}. Threshold effects and
structure formulas are treated in
\cite{Yajima2006,FincoYajima2006,Beceanu2014,
  BeceanuSchlag2020,DMSY2018,CMY2019}.
In one dimension, the corresponding $L^p$ bounds for
$1<p<\infty$ were established by 
Artbazar--Yajima \cite{ArtbazarYajima2000}, 
Weder \cite{Weder1999},
and
D'Ancona--Fanelli \cite{DAnconaFanelli2006}, under short
range assumptions on the potential.

Concerning the critical potentials in \eqref{eq:general-H},
dispersive and Strichartz estimates for
such operators have also been obtained directly from their
spectral representations \cite{FFFP2013,FZZ2022,GYZZ2022}.
Closer precedents of our results are the inverse square 
transplantation
and Sobolev bounds of Miao--Su--Zheng \cite{MSZ2023} and the
electromagnetic intertwining theory of
Fanelli--Su--Wang--Zhang--Zheng \cite{FSWZZ2026}.
The latter constructs a spectral intertwiner $S$
(without phase) between $\mathcal L_{\mathbf A,0}$ and
$\mathcal L_{\mathbf A,a}$, with the magnetic potential
fixed. Its planar $L^p$ bounds allow $W^{1,\infty}$
coefficients and, at flux $\alpha=\frac 12$, impose a symmetry
condition. In our first result, we handle
the full wave operator from $\mathcal L_{\mathbf A,a}$
to $-\Delta$, under moderate smoothness assumptions
and for $\alpha\neq \frac 12$. To make the paper self contained,
we include a proof of the required transplantation bound in
Appendix~\ref{app:transplantation}. 

Our first result is the following.

\begin{theorem}[General electromagnetic bounds]
\label{thm:general}
  Let $a\in W^{3,\infty}(\T;\R)$, $b\in C^{1}(\T;\R)$, let
  $\mathbf A(\theta)=b(\theta) e_\theta$, and assume 
  $B_{\mathbf A,a}\ge0$ in \eqref{eq:general-angular} and
  $\alpha\notin\frac12\Z$. The wave operators
  $W_\pm(\mathcal L_{\mathbf A,a},-\Delta)$ exist and are
  unitary on $L^2(\R^2)$. For every $1<p<\infty$, they and
  their adjoints extend boundedly to $L^p(\R^2)$ and are
  mutually inverse there. 
  % The same conclusions hold for
  % $W_\pm(\mathcal L_{\mathbf A,a},\mathcal L_{\mathbf A,0})$.
\end{theorem}

\begin{remark}[Half flux]
  The case $\alpha=1/2$ is excluded from the previous result.
  Nevertheless, under the
  hypotheses of Theorem~\ref{thm:general} plus the
  symmetry $a(\pi-\theta)=a(\pi+\theta)$,
  $0\le\theta\le\pi$, combining
  Corollary~\ref{cor:AB-unweighted} below with
  \cite[Theorem~2.13]{FSWZZ2026} gives also at half flux a unitary
  spectral intertwiner from $-\Delta$ to
  $\mathcal L_{\mathbf A,a}$, bounded together with its
  inverse on every $L^p$, $1<p<\infty$. Thus spectral
  similarity and transference remain available, without
  identifying a true wave operator; see
  Remark~\ref{rem:half-flux-intertwining}.
\end{remark}

We next specialize to the Aharonov--Bohm model.
For constant $b=\alpha$ and $a=0$, the Friedrichs
Hamiltonian is
\begin{equation}\label{eq:Halpha}
  H_\alpha
  =\left(-i\nabla+\alpha\frac{(-x_2,x_1)}{|x|^2}\right)^2,
  \qquad \alpha\in\R.
\end{equation}
Set
\begin{equation}\label{eq:waveop}
  W_\pm=W_\pm(H_\alpha,H_0),\qquad H_0=-\Delta.
\end{equation}
These operators exist and are unitary on $L^2$, see
\cite{Ruijsenaars1983}; a proof with our conventions is given
in Section~\ref{sec:prelim}. We retain all $\alpha\in \mathbb{R}$,
since integer gauge changes conjugate the magnetic
Hamiltonian but do not leave the free comparison operator
fixed.

For $\beta\in\R$, put
\begin{equation*}
  L^p_\beta=L^p(\R^2,|x|^\beta\,dx),\qquad
  \rho_\alpha=\operatorname{dist}(\alpha,\Z).
\end{equation*}

\begin{theorem}[Friedrichs weighted estimates]
  \label{thm:weighted}
  Let $\alpha\in\R$, $1<p<\infty$, and $\beta\in\R$.
  \begin{enumerate}
  \item[(i)] If $\alpha=0$, then $W_\pm=W_\pm^*=I$ on $L^p_\beta$
             for every $\beta\in\R$.
  \item[(ii)] If $\alpha\ne0$, then
  \begin{equation}\label{eq:weighted-W-range}
    W_\pm:L^p_\beta\longrightarrow L^p_\beta
    \quad\text{is bounded if and only if}\quad
    -2-p\rho_\alpha<\beta<2(p-1),
  \end{equation}
  whereas
  \begin{equation}\label{eq:weighted-Wstar-range}
    W_\pm^*:L^p_\beta\longrightarrow L^p_\beta
    \quad\text{is bounded if and only if}\quad
    -2<\beta<2(p-1)+p\rho_\alpha.
  \end{equation}
Consequently, both $W_\pm$ and $W_\pm^*$ are bounded on
  $L^p_\beta$ exactly for
  \begin{equation}\label{eq:weighted-similarity-range}
    -2<\beta<2(p-1),
  \end{equation}
  \end{enumerate}
  For every $\alpha\ne0$, the operators $W_\pm$ and $W_\pm^*$
  are unbounded on $L^1(\R^2)$ and on $L^\infty(\R^2)$.
\end{theorem}

In particular, in the unweighted case we reobtain

\begin{corollary}[Friedrichs $L^{p}$ bounds]
\label{cor:AB-unweighted}
  For every $\alpha\in\R$ and $1<p<\infty$, the operators
  $W_\pm,W_\pm^*$ are bounded on $L^p(\R^2)$, with constants
  locally bounded in $\alpha$. 
\end{corollary}

We now turn to the effect of the boundary condition.
For $0<\alpha<1$, the Krein realization replaces the regular
behaviours $r^\alpha$ and $r^{1-\alpha}$ by the singular
behaviours $r^{-\alpha}$ and $r^{-(1-\alpha)}$ in the two
critical angular channels, while all other channels are unchanged,
see \eqref{eq:Krein-def} below for precise definitions.
The new effect is a reduced parameter range for the estimates,
even in unweighted $L^{p}$ spaces:

\begin{theorem}[Krein weighted estimates]\label{thm:krein}
Let $0<\alpha<1$, let $H_\alpha^{\rm K}$ be the Krein extension
defined in \eqref{eq:Krein-def}, and set
\begin{equation*}
  \eta_\alpha=\max\{\alpha,1-\alpha\}=1-\rho_\alpha.
\end{equation*}
Its wave operators
$$W_\pm^{\rm K}=\slim_{t\to\pm\infty}e^{itH_\alpha^{\rm
  K}}e^{-itH_0}$$
exist and are unitary on $L^2(\R^2)$. If $1<p<\infty$ and
$\beta\in\R$, then
\begin{equation}\label{eq:Krein-W-range}
  W_\pm^{\rm K}:L^p_\beta\longrightarrow L^p_\beta
  \quad\text{is bounded if and only if}\quad
  p\eta_\alpha-2<\beta<2(p-1).
\end{equation}
Moreover,
\begin{equation}\label{eq:Krein-Wstar-range}
  (W_\pm^{\rm K})^*:L^p_\beta\longrightarrow L^p_\beta
  \quad\text{is bounded if and only if}\quad
  -2<\beta<2(p-1)-p\eta_\alpha.
\end{equation}
Consequently, 
\begin{equation}\label{eq:Krein-unweighted}
  W_\pm^{\rm K}\text{ and }(W_\pm^{\rm K})^*\text{ are both
  bounded on }L^p(\R^2)
  \quad\Longleftrightarrow\quad
  \frac{2}{2-\eta_\alpha}<p<\frac{2}{\eta_\alpha}.
\end{equation}
They are inverse isomorphisms on $L^p_\beta$
precisely when
\begin{equation}\label{eq:Krein-similarity-range}
  p\eta_\alpha-2<\beta<2(p-1)-p\eta_\alpha.
\end{equation}
\end{theorem}

We recall that selfadjoint realizations of the AB Hamiltonian
and scattering formulas have been studied in
\cite{Ruijsenaars1983,AdamiTeta1998,Richard2009,
  PankrashkinRichard2011,Fermi2024},
and weighted Hankel transplantation estimates are classical
\cite{Stempak2002,NowakStempak2006}. In the normalization
\eqref{eq:hankel}, \cite[Corollary~1.4]{Stempak2002} gives
boundedness of $\HH_\nu\HH_\mu$ on
$L^p((0,\infty),r^{\beta+1}\,dr)$ whenever
$\nu,\mu>-1$, $1<p<\infty$, and
\begin{equation*}
  -2-p\nu<\beta<p(\mu+2)-2.
\end{equation*}
This includes the negative orders of the Krein channels.
Ciaurri \cite{Ciaurri2026} proves uniform weighted estimates
for shifted order families and an $\ell^2$ valued inequality,
giving mixed norm results with angular $L^2$.
Such estimates do not by themselves control the full spatial
$L^p$ norm when $p\ne2$.

\subsection{Organization of the proof}

Sections~\ref{sec:prelim}--\ref{sec:interior} develop the
spectral and Mellin representations and prove sufficient AB
weighted bounds. Their unweighted consequence is the
magnetic input for the general theorem in
Section~\ref{sec:electromagnetic}.
Sections~\ref{sec:endpoints} and \ref{sec:Krein} establish
sharpness, endpoint failure, and the Krein classification.
Section~\ref{sec:applications} collects transference and
change of flux. Appendix~\ref{app:transplantation} proves
the transplantation bound used in
Section~\ref{sec:electromagnetic} from the Mellin estimates
and angular lemmas already established in the paper.

\section{Spectral preliminaries and scattering formulas}
\label{sec:prelim}

In this section, we introduce some notation and preliminaries about spectral property and known scattering results. 

\subsection{Notation}

Polar coordinates on $\R^2$ are 
denoted by $x=(r\cos\theta,r\sin\theta)$.
For $f\in L^2(\R^2)$, we write
\begin{equation}\label{eq:angular}
  f(r,\theta)=(2\pi)^{-1/2}\sum_{k\in\Z}
    e^{ik\theta}f_k(r),\qquad
  f_k(r)=(2\pi)^{-1/2}\int_0^{2\pi}
    e^{-ik\theta}f(r,\theta)\dd\theta,
\end{equation}
so that $\|f\|_{L^2(\R^2)}^2=\sum_k\|f_k\|^2_{L^2(\R_+,r\,dr)}$;
the map $f\mapsto(f_k)_{k\in\Z}$ is unitary from $L^2(\R^2)$
onto $\bigoplus_{k}L^2(\R_+,r\,dr)$, and we denote by
\begin{equation}\label{eq:Pk}
  P_kf(r,\theta)=\frac1{2\pi}\int_0^{2\pi}
    e^{ik(\theta-\omega)}f(r,\omega)\dd\omega
  =(2\pi)^{-1/2}e^{ik\theta}f_k(r),
\end{equation}
the projection onto the $k$th angular mode; by Minkowski's
inequality $P_k$ is a contraction on every $L^p(\R^2)$, $1\le
p\le\infty$. For a fixed flux $\alpha\in\R$, we set
\begin{equation}\label{eq:mudef}
  \mu_k=|k|,\qquad \nu_k=|k+\alpha|,\qquad
  \delta_k=\tfrac{\pi}{2}(\mu_k-\nu_k),\qquad k\in\Z.
\end{equation}
The Hankel transform of real order $\gamma>-1$ is 
\begin{equation}\label{eq:hankel}
  (\HH_\gamma h)(\xi)=\int_0^\infty J_\gamma(r\xi)\,h(r)\,r\dd
  r,\qquad
  \xi\ge0,
\end{equation}
where $J_\gamma$ is the Bessel function of the first kind, and
the Mellin transform is 
\begin{equation}\label{eq:mellin}
  \Mell h(z)=\int_0^\infty h(r)\,r^{z-1}\dd r.
\end{equation}
Here $z\in \C$ if $h\in C^\infty_c(\R^+)$ or may be restricted to a suitable vertical strip, depending on the properties of $h$.
For $1\le p\le\infty$, the map
\begin{equation}\label{eq:Up}
  (\mathcal U_ph)(u)=e^{2u/p}h(e^u)
\end{equation}
is an isometric isomorphism of $L^p(\R_+,r\,dr)$ onto
$L^p(\R,du)$ (with $\mathcal U_\infty h(u)=h(e^u)$), and with
the Fourier transform $\widehat g(\tau)=\int_\R
e^{-iu\tau}g(u)\dd u$ we have
\begin{equation}\label{eq:mellin-fourier}
  \widehat{\mathcal U_ph}(\tau)=\Mell
  h\bigl(\tfrac2p-i\tau\bigr).
\end{equation}
Finally, for $a\in\R$ and $\zeta\in\C$, we use throughout the Gamma quotient
\begin{equation}\label{eq:Ga}
  G_a(\zeta)=\Gam{\zeta+a}{\zeta},\qquad \Re\zeta>0,\
  \Re(\zeta+a)>0.
\end{equation}
We use $C$ to denote a positive constant whose value may change from line
to line.

\subsection{Bessel functions and the Hankel transform}

We recall the facts about Bessel functions that will be used;
see \cite[Ch.~10]{NIST} and \cite{Watson1944}. For every real
$\gamma>-1$,
\begin{equation}\label{eq:bessel-bounds}
  |J_\gamma(x)|\le C_\gamma x^\gamma\quad(0<x\le1),\qquad
  |J_\gamma(x)|\le C_\gamma x^{-1/2}\quad(x\ge1),
\end{equation}
and more precisely, for $x\ge1$,
\begin{equation}\label{eq:bessel-asym}
  J_\gamma(x)=(2\pi x)^{-1/2}
  \Bigl\{e^{i(x-\pi\gamma/2-\pi/4)}
    +e^{-i(x-\pi\gamma/2-\pi/4)}\Bigr\}
  +R_\gamma(x),\qquad |R_\gamma(x)|\le C_\gamma x^{-3/2}.
\end{equation}
The Mellin transform of $J_\gamma$ is \cite[Eq.~10.22.43]{NIST}
\begin{equation}\label{eq:bessel-mellin}
  j_\gamma(w)\coloneq\int_0^\infty t^{w-1}J_\gamma(t)\dd t
  =2^{w-1}\,\Gam{\frac{\gamma+w}{2}}
    {\frac{\gamma-w+2}{2}},\qquad
  -\gamma<\Re w<\tfrac32,
\end{equation}
Within this strip, the integral is absolutely convergent when
$\Re w<\frac12$ and converges as an improper integral when $\Re
w\ge\frac12$.

The Hankel transform $\HH_\gamma$, $\gamma>-1$, initially
defined on $C_c^\infty(0,\infty)$, extends to a unitary
involution of $L^2(\R_+,r\,dr)$: $\HH_\gamma\HH_\gamma=I$ and
$\HH_\gamma^*=\HH_\gamma$; see \cite[\S8.18]{Titchmarsh1948} for
nonnegative orders and
\cite[Propositions~4.5--4.6]{DerezinskiRichard2017} for the full
range. For $h\in C_c^\infty(0,\infty)$ and
\begin{equation}\label{eq:Lgamma}
  L_\gamma=-\partial_r^2-\frac1r\partial_r+\frac{\gamma^2}{r^2},
\end{equation}
Bessel's equation $L_\gamma
J_\gamma(\cdot\,\xi)=\xi^2J_\gamma(\cdot\,\xi)$ and two
integrations by parts give
\begin{equation}\label{eq:hankel-L}
  \HH_\gamma(L_\gamma h)(\xi)=\xi^2\,\HH_\gamma h(\xi),\qquad
  h\in C_c^\infty(0,\infty).
\end{equation}
We define the selfadjoint operator
\begin{equation}\label{eq:hgamma}
  h_\gamma=\HH_\gamma\,M_{\xi^2}\,\HH_\gamma
  \quad\text{on}\quad L^2(\R_+,r\,dr),
\end{equation}
where $M_{\xi^2}$ is multiplication by $\xi^2$ on its maximal
domain. By \eqref{eq:hankel-L}, $h_\gamma$ extends $L_\gamma$ on
$C_c^\infty(0,\infty)$. For $\gamma\ge0$, this is the
Friedrichs extension of $L_\gamma$. When $\gamma\ge1$, the
endpoint $r=0$ is in the limit point case, so no boundary
condition is imposed there. When $0\le\gamma<1$, the endpoint
is in the limit circle case and the Friedrichs domain selects
the regular behaviour $O(r^\gamma)$.

Now suppose that $0<\gamma<1$. The differential expressions
$L_\gamma$ and $L_{-\gamma}$ then coincide, but their
Friedrichs and Krein realizations are different. Every $v$ in
the maximal domain has an expansion
\begin{equation*}
  v(r)=a r^{-\gamma}+b r^\gamma+o(r),\qquad r\to0.
\end{equation*}
The Friedrichs domain of $h_\gamma$ imposes $a=0$, while the
Krein domain of $h_{-\gamma}$ imposes $b=0$. Thus the Krein
condition is a cancellation condition on the regular
coefficient; note that it is not sufficient to impose
$v(r)=O(r^{-\gamma})$ as $r\to0$, since
$r^\gamma=O(r^{-\gamma})$ as well, and the growth estimate
allows both terms and does not identify
the Krein domain.
These assertions follow after the unitary transform
$h(r)\mapsto r^{1/2}h(r)$ from the homogeneous inverse square
realizations in \cite[\S2.3]{DerezinskiRichard2017}; see also
\cite[Sections~3--4]{Fermi2024} for the Aharonov--Bohm
extensions.

\subsection{The Aharonov--Bohm Hamiltonian and
its wave operators}
In the angular decomposition \eqref{eq:angular} the free
Laplacian acts on the $k$th mode as $L_{|k|}$, and the
classical formula for the Fourier transform of
$e^{ik\theta}g(r)$ \cite[Ch.~IV, Thm.~3.10]{SteinWeiss1971}
shows that $-\Delta$ acts on the $k$th mode as $h_{|k|}$.
Likewise, the differential expression \eqref{eq:Halpha} acts on
$e^{ik\theta}g(r)$ as $e^{ik\theta}L_{|k+\alpha|}g$, and the
Friedrichs extension of \eqref{eq:Halpha} on
$C_c^\infty(\R^2\setminus\{0\})$ is
\begin{equation}\label{eq:Halpha-def}
  H_\alpha=\bigoplus_{k\in\Z}h_{\nu_k},\qquad
  \text{i.e.}\quad (H_\alpha f)_k=h_{\nu_k}f_k,
  \qquad
  \nu_{k}=|k+\alpha|,
\end{equation}
see \cite[\S2 and Appendix~A]{Ruijsenaars1983},
\cite[\S\S2--3]{PankrashkinRichard2011}; the only modes affected
by the choice of the selfadjoint extension are those with
$\nu_k<1$. In particular $H_\alpha\ge0$, $H_\alpha$ is unitarily
equivalent to $\bigoplus_kM_{\xi^2}$ and hence has purely
absolutely continuous spectrum $[0,\infty)$, and
\begin{equation}\label{eq:evolutions}
  (e^{itH_\alpha}f)_k=\HH_{\nu_k}e^{it\xi^2}\HH_{\nu_k}f_k,\qquad
  (e^{-itH_0}f)_k=\HH_{\mu_k}e^{-it\xi^2}\HH_{\mu_k}f_k,\qquad
  \mu_{k}=|k|.
\end{equation}

For the Krein extension we restrict to $0<\alpha<1$, which
identifies the two critical modes without an additional gauge
shift. For $k\in\Z$, set
\begin{equation}\label{eq:Krein-orders}
  \lambda_k=
  \begin{cases}
    -\alpha,&k=0,\\
    \alpha-1,&k=-1,\\
    \nu_k,&k\notin\{0,-1\}.
  \end{cases}
\end{equation}
Thus $-1<\lambda_k$ for every $k$, and the Krein extension of
the minimal Aharonov--Bohm operator is
\begin{equation}\label{eq:Krein-def}
  H_\alpha^{\rm K}=\bigoplus_{k\in\Z}h_{\lambda_k}.
\end{equation}
Indeed, the two critical Friedrichs summands $h_\alpha$ and
$h_{1-\alpha}$ are replaced by $h_{-\alpha}$ and
$h_{-(1-\alpha)}$, while every other summand is essentially
selfadjoint. Formula \eqref{eq:hgamma} shows that
$H_\alpha^{\rm K}$ is nonnegative and has purely absolutely
continuous spectrum $[0,\infty)$.

The following partial wave formula for the wave operators was
proved by Ruijsenaars \cite[Theorem~A1]{Ruijsenaars1983} for
$0<\alpha<1$. Since the proof for general $\alpha$ is short, we
include it.

\begin{proposition}[Friedrichs partial wave formula]
\label{prop:partial-wave}
Let $\alpha\in\R$. The strong limits \eqref{eq:waveop} exist,
$W_\pm$ are unitary on $L^2(\R^2)$, and
\begin{equation}\label{eq:partial-wave}
  (W_\pm f)_k=W_{\pm,k}f_k,\qquad
  W_{\pm,k}=e^{\mp i\delta_k}\,\HH_{\nu_k}\HH_{\mu_k},\qquad
  k\in\Z,
\end{equation}
with $\mu_k,\nu_k,\delta_k$ as in \eqref{eq:mudef}. Moreover
$W_\pm$ commute with rotations, hence with every $P_k$, and
$(W_\pm^*f)_k=e^{\pm i\delta_k}\HH_{\mu_k}\HH_{\nu_k}f_k$.
\end{proposition}

\begin{proof}
By \eqref{eq:evolutions} and the unitarity of $\HH_{\mu_k}$,
\eqref{eq:partial-wave} follows from the scalar statement: for
$\mu,\nu\ge0$ and $d=\frac\pi2(\mu-\nu)$, the following limit
holds strongly on $L^2(\R_+,r\,dr)$:
\begin{equation}\label{eq:scalar}
  \slim_{t\to\pm\infty}
  \HH_\nu e^{it\xi^2}\HH_\nu\,\HH_\mu e^{-it\xi^2}\HH_\mu
  =e^{\mp id}\,\HH_\nu\HH_\mu.
\end{equation}
Right composition with the involution $\HH_\mu$ reduces
 \eqref{eq:scalar} to
\begin{equation*}
  \slim_{t\to\pm\infty}
  \HH_\nu e^{it\xi^2}\HH_\nu\HH_\mu e^{-it\xi^2}
  =e^{\mp id}\HH_\nu.
\end{equation*}
Applying the unitary operator $\HH_\nu e^{-it\xi^2}\HH_\nu$ to
the difference shows that it suffices to prove, on the dense
union of the spaces $C_c^\infty((\varepsilon,N))$,
\begin{equation}\label{eq:scalar-norm}
  \lim_{t\to\pm\infty}\bigl\|\Phi_t\bigr\|_{L^2(r\,dr)}=0,\qquad
  \Phi_t(r)\coloneq\int_0^\infty\bigl[J_\mu(r\xi)-e^{\mp
  id}J_\nu(r\xi)\bigr]
  e^{-it\xi^2}F(\xi)\,\xi\dd\xi,
\end{equation}
for every $F\in C_c^\infty((\varepsilon,N))$,
$0<\varepsilon<N<\infty$. We treat $t\to+\infty$; the case
$t\to-\infty$ is identical after exchanging the roles of the two
exponentials in \eqref{eq:bessel-asym}. 

Now fix $A\ge1/\varepsilon$. On $0<r<A$: for each $r$, $\Phi_t(r)\to0$ as $t\to\infty$ by the
Riemann--Lebesgue lemma (substitute $\lambda=\xi^2$), and
$|\Phi_t(r)|\le C_F$ uniformly, since $J_\mu,J_\nu$ are bounded
on $[0,AN]$. By dominated convergence,
$\int_0^A|\Phi_t|^2r\,dr\to0$.

On $r>A$ we have $r\xi\ge A\varepsilon\ge1$ on the support of
$F$, and we insert \eqref{eq:bessel-asym}. The coefficient of
$e^{i(r\xi-\pi/4)}$ in $J_\mu(r\xi)-e^{-id}J_\nu(r\xi)$ is
$$(2\pi
r\xi)^{-1/2}\bigl(e^{-i\pi\mu/2}-e^{-id}e^{-i\pi\nu/2}\bigr)=0$$
by the choice of $d$. Hence for some constant $\gamma$
\begin{equation*}
  \begin{split}
    \Phi_t(r)={}&\gamma\,r^{-1/2}\int_0^\infty
    e^{-i(t\xi^2+r\xi)}F(\xi)\xi^{1/2}\dd\xi
    +\int_0^\infty\bigl[R_\mu(r\xi)-e^{-id}R_\nu(r\xi)\bigr]
    e^{-it\xi^2}F(\xi)\xi\dd\xi\\
    =:{}&\Phi^{(1)}_t(r)+\Phi^{(2)}_t(r).
  \end{split}
\end{equation*}
By \eqref{eq:bessel-asym},
$|\Phi^{(2)}_t(r)|\le C_Fr^{-3/2}$, so that
$\int_A^\infty|\Phi^{(2)}_t|^2r\,dr\le C_F/A$ uniformly in $t$.
For $\Phi^{(1)}_t$ the phase $\phi(\xi)=t\xi^2+r\xi$ satisfies
$\phi'(\xi)=2t\xi+r\ge2t\varepsilon+r$ and $\phi''=2t$ on the
support of $F$; integrations by parts twice, using
$e^{-i\phi}=\frac{i}{\phi'}\partial_\xi e^{-i\phi}$ and
$|\phi''/\phi'^{\,2}|\le C_\varepsilon/\phi'$, give
$$|\Phi^{(1)}_t(r)|\le C_{F}\,r^{-1/2}(t\varepsilon+r)^{-2},$$
whence 
$$\int_A^\infty|\Phi^{(1)}_t|^2r\,dr\le
C_F\int_A^\infty(t\varepsilon+r)^{-4}dr \le
C_F(t\varepsilon)^{-3}\to0.$$

Altogether $\limsup_{t\to\infty}\|\Phi_t\|^2_{L^2(r\,dr)}\le
C_F/A$ for every $A$, which proves \eqref{eq:scalar-norm}. Thus
the channel limits exist and are the unitary operators
\eqref{eq:partial-wave}. For $f$ with finitely many nonzero
modes the limit \eqref{eq:waveop} therefore exists, and since
$e^{itH_\alpha}e^{-itH_0}$ are unitary, the limit extends to all
of $L^2(\R^2)$ and is unitary. The rotation invariance and the
formula for $W_\pm^*$ are immediate from \eqref{eq:partial-wave}
and $\HH_\gamma^*=\HH_\gamma$.
\end{proof}

\begin{proposition}[Krein partial wave formula]
\label{prop:Krein-partial-wave}
Let $0<\alpha<1$ and define
\begin{equation*}
  \delta_k^{\rm K}=\frac\pi2(\mu_k-\lambda_k).
\end{equation*}
The wave operators for $H_\alpha^{\rm K}$ exist, are unitary on
$L^2(\R^2)$, commute with rotations, and satisfy
\begin{equation}\label{eq:Krein-partial-wave}
  (W_\pm^{\rm K}f)_k
  =e^{\mp i\delta_k^{\rm K}}\HH_{\lambda_k}\HH_{\mu_k}f_k,
  \qquad
  ((W_\pm^{\rm K})^*f)_k
  =e^{\pm i\delta_k^{\rm K}}\HH_{\mu_k}\HH_{\lambda_k}f_k.
\end{equation}
\end{proposition}

\begin{proof}
Under the transform $h(r)\mapsto r^{1/2}h(r)$, the
scalar wave operator formula for the homogeneous inverse square
realizations of any two real orders $\lambda,\mu>-1$ is
\begin{equation*}
  \slim_{t\to\pm\infty}e^{ith_\lambda}e^{-ith_\mu}
  =e^{\mp i\frac\pi2(\mu-\lambda)}\HH_\lambda\HH_\mu;
\end{equation*}
see \cite[Theorem~4.10 and Eqs.~(4.22)--(4.24)]
{DerezinskiRichard2017}.
Apply this with $(\lambda,\mu)=(\lambda_k,\mu_k)$ in the two
critical channels.
In every other channel the assertion is
Proposition~\ref{prop:partial-wave}. The convergence of the
orthogonal sum follows first for vectors with finitely many
angular modes and then for all of $L^2$ because every
approximating operator is unitary. The formula for the adjoint
follows from the selfadjoint involution property of each Hankel
transform.
\end{proof}

\subsection{General angular scattering}

For a constant magnetic coefficient, write the angular
operator and its Friedrichs Hamiltonian as
\begin{equation}\label{eq:em-H}
  B_{\alpha,a}=(D_\theta+\alpha)^2+a(\theta),
  \qquad
  H_{\alpha,a}=-\partial_r^2-r^{-1}\partial_r
  +r^{-2}B_{\alpha,a}.
\end{equation}
Thus $H_{\alpha,0}=H_\alpha$. All angular operators below have
periodic boundary conditions on $\T$.

We now separate the general existence question from the
$L^p$ estimates. Let $B$ be a nonnegative selfadjoint
operator on $L^2(\T)$ with an orthonormal eigenbasis
$\{\psi_n\}$ and
eigenvalues $s_n^2$. Define the unitary involution
\begin{equation}\label{eq:em-K}
  \mathcal K_B\left(\sum_n f_n(r)\psi_n(\theta)\right)
  =\sum_n(\HH_{s_n}f_n)(r)\psi_n(\theta).
\end{equation}
The sums converge in $L^2(\R^2)$. The corresponding Friedrichs
operator $H_B$ is the orthogonal sum of the radial operators
$h_{s_n}$ and satisfies $H_B=\mathcal K_B r^2\mathcal K_B$.
In particular, its spectrum is purely absolutely continuous.

\begin{proposition}\label{prop:em-L2}
  For two such angular operators $B_0,B_1$, the ordinary
  wave operators exist and are unitary. More precisely,
  \begin{equation}\label{eq:em-L2}
    W_\pm(H_{B_1},H_{B_0})
    =\mathcal K_{B_1}e^{\pm i\pi\sqrt{B_1}/2}
    e^{\mp i\pi\sqrt{B_0}/2}\mathcal K_{B_0}.
  \end{equation}
  No relation between the two angular eigenbases is needed.
\end{proposition}

\begin{proof}
  Choose normalized eigenvectors $e_m$ of $B_0$ and $\phi_n$
  of $B_1$, with eigenvalues $\mu_m^2$ and $\nu_n^2$, and set
  $C_{nm}=\int_\T e_m\overline{\phi_n}\,d\theta$. For
  $f\in L^2(\R_+,r\,dr)$, separation of variables gives
  \begin{equation*}
    e^{itH_{B_1}}e^{-itH_{B_0}}(fe_m)
    =\sum_n C_{nm}
    \bigl(e^{ith_{\nu_n}}e^{-ith_{\mu_m}}f\bigr)\phi_n.
  \end{equation*}
  On each term, the scalar limit \eqref{eq:scalar} is
  \begin{equation*}
    \slim_{t\to\pm\infty}
    e^{ith_{\nu_n}}e^{-ith_{\mu_m}}
    =e^{\pm i\pi(\nu_n-\mu_m)/2}
    \HH_{\nu_n}\HH_{\mu_m}.
  \end{equation*}
  The squared norm of the tail after $n=N$ is exactly
  $\|f\|_2^2\sum_{n>N}|C_{nm}|^2$, uniformly in $t$; the
  limiting series has the same tail bound. Truncate the sum,
  take the scalar limits, and then let $N\to\infty$.
  This proves convergence on a single input angular mode,
  and hence on finite sums. Unitarity of the evolutions and
  density extend the limit to all of $L^2$.

  The resulting matrix formula is precisely
  \eqref{eq:em-L2}. Its four factors are unitary, so the
  limit is unitary and its range is the whole space. Reversing
  $B_0,B_1$ gives its adjoint and inverse.
\end{proof}

This formula explains why the spectral intertwiner (without
phase) is not, in general, an ordinary wave operator. Even if the
angular eigenbases coincide, scattering introduces the phase
$e^{\pm i\pi(\nu_k-\mu_k)/2}$. For example, take
$\alpha=1/4$ and $a=1/2$. In the mode $k=0$, the magnetic
order is $1/4$ and the perturbed order is $3/4$, so
\begin{equation*}
  W_\pm(H_{1/4,1/2},H_{1/4})\big|_{k=0}
  =e^{\pm i\pi/4}\HH_{3/4}\HH_{1/4}.
\end{equation*}
For nonconstant $a$, the overlap coefficients $C_{nm}$ also
mix the angular modes.

\section{The Mellin symbol}\label{sec:mellin}

The operators $\HH_\nu\HH_\mu$ commute with dilations, hence are
Mellin multipliers. The following proposition identifies the
multiplier and, crucially for $p\ne2$, the strip of analyticity
on which the identity holds.

\begin{proposition}[Mellin symbol]\label{prop:symbol}
Let $\mu\ge0$ and let $\nu>-1$ be real. Define
\begin{equation}\label{eq:symbol}
  m_{\nu,\mu}(z)
  =\frac{\Gamma\bigl(\frac{\nu+z}{2}\bigr)
    \Gamma\bigl(\frac{\mu+2-z}{2}\bigr)}
  {\Gamma\bigl(\frac{\nu+2-z}{2}\bigr)
    \Gamma\bigl(\frac{\mu+z}{2}\bigr)},
  \qquad z\in S_{\nu,\mu}\coloneq\{-\nu<\Re z<\mu+2\}.
\end{equation}
Then $m_{\nu,\mu}$ is holomorphic and bounded on every closed
substrip of $S_{\nu,\mu}$. For $h\in C_c^\infty(0,\infty)$, the
function $f=\HH_\nu\HH_\mu h$ is smooth on $(0,\infty)$,
satisfies $|f(r)|\le C_{c,h}\,r^{-c}$ for every
$c\in(-\nu,\mu+2)$, and
\begin{equation}\label{eq:symbol-id}
  \Mell(\HH_\nu\HH_\mu h)(z)=m_{\nu,\mu}(z)\,\Mell h(z),\qquad
  z\in S_{\nu,\mu},
\end{equation}
with absolutely convergent integrals. Consequently, for $1\le
p\le\infty$ with $2/p\in S_{\nu,\mu}$,
\begin{equation}\label{eq:symbol-fourier}
  \widehat{\mathcal U_p(\HH_\nu\HH_\mu h)}(\tau)
  =m_{\nu,\mu}\bigl(\tfrac2p-i\tau\bigr)\,\widehat{\mathcal
  U_ph}(\tau),
  \qquad\tau\in\R.
\end{equation}
\end{proposition}

\begin{remark}
The strip $S_{\nu,\mu}$ always contains the line $\Re z=1$, and
it contains the interval $(0,2)$ when $\mu,\nu\ge0$. Moreover,
$2\in S_{\nu,\mu}$ if and only if $\mu>0$, and $0\in
S_{\nu,\mu}$ if and only if $\nu>0$. On the line $\Re z=1$ the
symbol has modulus one, as it must, since $\HH_\nu\HH_\mu$ is
unitary. Writing $z=1-i\tau$ in \eqref{eq:symbol} one recovers
the wave operator representation of
\cite{Richard2009,PankrashkinRichard2011} after inserting the
phase in \eqref{eq:partial-wave}.
\end{remark}

\begin{proof}
\emph{Step 1: analyticity and bounds.} In terms of
\eqref{eq:Ga}, interpreting the Gamma quotients meromorphically
where necessary,
\begin{equation}\label{eq:symbol-G}
  m_{\nu,\mu}(z)=\frac{G_a(\zeta_1)}{G_a(\zeta_2)},\qquad
  a=\frac{\nu-\mu}2,\quad \zeta_1=\frac{\mu+z}{2},\quad
  \zeta_2=\frac{\mu+2-z}{2}.
\end{equation}
For $z\in S_{\nu,\mu}$ the two numerator arguments $(\nu+z)/2$
and $(\mu+2-z)/2$ have positive real part. They therefore avoid
the poles of $\Gamma$, while the reciprocal Gamma factors in the
denominator are entire; hence $m_{\nu,\mu}$ is holomorphic on
$S_{\nu,\mu}$. Notice that the denominator arguments need not
have positive real part, and their poles simply give zeros of
$m_{\nu,\mu}$. On a closed substrip, boundedness for $|\Im
z|\le1$ follows by compactness. For $|\Im z|\ge1$, the uniform
vertical Stirling asymptotic \cite[Eq.~5.11.9]{NIST} applies to
all four factors. The exponential factors cancel, and so do the
total powers of $|\Im z|$, since the sums of the numerator and
denominator real parts agree. Thus $|m_{\nu,\mu}(z)|\le C$
uniformly on every closed substrip.

\emph{Step 2: the first Hankel transform.} Let $h\in
C_c^\infty(0,\infty)$ and $g=\HH_\mu h$. By \eqref{eq:hankel-L},
$\xi^{2N}g=\HH_\mu(L_\mu^Nh)$ is bounded for every $N$, so $g$
is smooth and rapidly decreasing at infinity, and by the first
bound in \eqref{eq:bessel-bounds}, $|g(\xi)|\le C\xi^\mu$. Hence
$\Mell g(w)$ converges absolutely and is holomorphic for $\Re
w>-\mu$. For $-\mu<\Re w<\frac12$ the double integral
$\int_0^\infty\int_0^\infty\xi^{\Re
w-1}|J_\mu(r\xi)||h(r)|\,r\,dr\,d\xi$ is finite by
\eqref{eq:bessel-bounds}, and Fubini's theorem together with
\eqref{eq:bessel-mellin} gives
\begin{equation}\label{eq:Mg}
  \Mell g(w)=\int_0^\infty h(r)\,r^{1-w}\Bigl(\int_0^\infty
  t^{w-1}J_\mu(t)\dd t\Bigr)dr
  =j_\mu(w)\,\Mell h(2-w).
\end{equation}
The right hand side is holomorphic for $\Re w>-\mu$ ($\Mell h$
is entire and $\Gamma(\frac{\mu+w}2)$ has its poles at
$w=-\mu-2j$), so \eqref{eq:Mg} holds for all $\Re w>-\mu$.

\emph{Step 3: a nonnegative target order.} Assume first that
$\nu\ge0$ and let $f=\HH_\nu g$. For $-\nu<\Re z<\frac12$ the
double integral
\begin{equation*}
    \int_0^\infty\!\int_0^\infty
    r^{\Re z-1}|J_\nu(r\xi)||g(\xi)|\xi\dd\xi\dd r
    =\int_0^\infty |g(\xi)|\xi^{1-\Re z}
    \left(\int_0^\infty t^{\Re z-1}|J_\nu(t)|\dd t\right)\dd\xi
\end{equation*}
is finite: the inner integral converges by
\eqref{eq:bessel-bounds}, and the outer one converges at $0$
because $|g(\xi)|\le C\xi^\mu$ and $\Re z<\frac12<\mu+2$, and at
infinity because $g$ is rapidly decreasing. Fubini's theorem,
\eqref{eq:bessel-mellin} and \eqref{eq:Mg} yield, with
absolutely convergent integrals,
\begin{equation}\label{eq:Mf-left}
\begin{split}
  \Mell f(z)&=j_\nu(z)\,\Mell g(2-z)\\&=j_\nu(z)\,j_\mu(2-z)\,\Mell
  h(z)
  =m_{\nu,\mu}(z)\,\Mell h(z),\quad -\nu<\Re z<\tfrac12.
  \end{split}
\end{equation}

\emph{Step 4: continuation to the strip.} Put
$F(z)=m_{\nu,\mu}(z)\Mell h(z)$. Integrating by parts $N$ times
in \eqref{eq:mellin}, $\Mell
h(z)=\frac{(-1)^N}{z(z+1)\cdots(z+N-1)}\int_0^\infty
h^{(N)}(r)r^{z+N-1}dr$, so $\Mell h$ is rapidly decreasing on
vertical lines, uniformly on compact $\Re z$ intervals; by
Step~1, so is $F$ on closed substrips of $S_{\nu,\mu}$. For
$c\in(-\nu,\mu+2)$ define
\begin{equation*}
  \tilde f_c(r)=\frac1{2\pi i}\int_{\Re z=c}F(z)\,r^{-z}\dd z.
\end{equation*}
By Cauchy's theorem and the decay of $F$, $\tilde f_c$ does not
depend on $c$; call it $\tilde f$. It is smooth on $(0,\infty)$
and $|\tilde f(r)|\le C_cr^{-c}$ for every $c\in(-\nu,\mu+2)$.
On the other hand, for $c\in(-\nu,\frac12)$, Step~3 shows that
$r^cf(r)$ is integrable with respect to $dr/r$, i.e.\ $u\mapsto
e^{cu}f(e^u)$ belongs to $L^1(\R)$, and that its Fourier
transform is $\tau\mapsto F(c-i\tau)$, which is in $L^1(\R)$.
Fourier inversion for $L^1$ functions with integrable transform
gives $f=\tilde f_c$ almost everywhere, hence everywhere, as
both sides are continuous ($f$ is continuous by dominated
convergence). Therefore $f=\tilde f$. For a fixed
$c\in(-\nu,\mu+2)$, choose $d_-,d_+\in(-\nu,\mu+2)$ with
$d_-<c<d_+$. The estimate with exponent $d_-$ on $(0,1)$ and the
estimate with exponent $d_+$ on $(1,\infty)$ show that $r^cf(r)$
is integrable with respect to $dr/r$. It is the inverse Fourier
transform of the integrable function $F(c-i\cdot)$; thus its
Fourier transform is $F(c-i\cdot)$, which is
\eqref{eq:symbol-id} on the line $\Re z=c$. Finally
\eqref{eq:symbol-fourier} is \eqref{eq:symbol-id} rewritten
through \eqref{eq:mellin-fourier}.

\emph{Step 5: a negative target order.} It remains to consider
$-1<\nu<0$. Let $f=\HH_\nu\HH_\mu h$. The scalar dilation
formula \cite[Eq.~(4.24)]{DerezinskiRichard2017}, after removing
its scattering phase and applying the Liouville transform
$u \mapsto r^{1/2}u(r)$, says precisely that
\begin{equation}\label{eq:negative-L2-symbol}
  \widehat{\mathcal U_2f}(\tau)
  =m_{\nu,\mu}(1-i\tau)\widehat{\mathcal U_2h}(\tau).
\end{equation}
Put again $F(z)=m_{\nu,\mu}(z)\Mell h(z)$. By Step~1 and
integration by parts in $\Mell h$, $F$ is rapidly decreasing on
vertical lines, uniformly on closed substrips of $S_{\nu,\mu}$.
Mellin--Plancherel and \eqref{eq:negative-L2-symbol} therefore
identify $f$ in $L^2$ with
\begin{equation*}
  \widetilde f(r)=\frac1{2\pi i}\int_{\Re z=1}F(z)r^{-z}\dd z.
\end{equation*}
The integral defining $\HH_\nu\HH_\mu h(r)$ is absolutely
convergent for each $r>0$: near zero its integrand is
$O(\xi^{\mu+\nu+1})$, and at infinity $\HH_\mu h$ is rapidly
decreasing. Compact truncations of $\HH_\mu h$ identify this
integral with the $L^2$ Hankel transform, and by dominated
differentiation on compact $r$ intervals we see that $f$ is
smooth there. Thus the $L^2$ identity with the continuous
function $\widetilde f$ holds everywhere.

Since $F$ is holomorphic and rapidly decreasing on every closed
substrip, Cauchy's theorem moves the last contour to any line
$c\in(-\nu,\mu+2)$ without changing its value. Consequently
$|f(r)|\le C_{c,h}r^{-c}$ throughout this strip. Given a fixed
$c$, choose $d_-<c<d_+$ in the same strip. The bounds with $d_-$
near zero and $d_+$ near infinity show that $r^cf(r)$ is
integrable with respect to $dr/r$. Fourier inversion on the line
$\Re z=c$ now gives \eqref{eq:symbol-id}, with absolute
convergence, and \eqref{eq:symbol-fourier} follows as before.
\end{proof}

We shall also need the behaviour of the Gamma quotient
\eqref{eq:Ga} in the right half plane, in a form uniform in the
parameter.

\begin{lemma}[Gamma quotients]\label{lem:gamma}
  Let $\mathcal A\subset\R$ be compact and let $\varepsilon>0$.
  For $a\in\mathcal A$ set $\Omega_a=\{\zeta\in\C:\
  \Re\zeta\ge\varepsilon,\ \Re(\zeta+a)\ge\varepsilon\}$. There is
  $C=C(\mathcal A,\varepsilon)$ such that, for every $a\in\mathcal
  A$ and $\zeta\in\Omega_a$,
  \begin{equation}\label{eq:G-size}
    C^{-1}|\zeta|^a\le|G_a(\zeta)|\le C|\zeta|^a.
  \end{equation}
  Furthermore, $\psi=\Gamma'/\Gamma$ satisfies
  \begin{equation}\label{eq:G-deriv}
    |\psi(\zeta+a)-\psi(\zeta)|\le\frac C{1+|\zeta|},\qquad
    |\psi'(\zeta+a)-\psi'(\zeta)|\le\frac C{(1+|\zeta|)^2}.
  \end{equation}
  Moreover, with the principal branch
  of $\zeta^{a}$, uniformly for $a\in\mathcal A$,
  \begin{equation}\label{eq:G-asym}
    G_a(\zeta)=\zeta^a\bigl(1+O(|\zeta|^{-1})\bigr),\qquad
    |\zeta|\to\infty,\ \zeta\in\Omega_a.
  \end{equation}
\end{lemma}

\begin{proof}
$\Omega_a$ is contained in $\{|\arg\zeta|<\pi/2\}$, where the
expansions $G_a(\zeta)=\zeta^a(1+O(|\zeta|^{-1}))$,
$\psi(\zeta)=\log\zeta-\frac1{2\zeta} +O(|\zeta|^{-2})$ and
$\psi'(\zeta)=\frac1\zeta+O(|\zeta|^{-2})$ hold uniformly as
$|\zeta|\to\infty$ and $a$ ranges over a compact set
\cite[Eqs.~5.11.2, 5.11.12, 5.15.8]{NIST}. This gives
\eqref{eq:G-asym}, and also
\eqref{eq:G-size}--\eqref{eq:G-deriv} for $|\zeta|\ge R(\mathcal
A,\varepsilon)$, since
$\psi(\zeta+a)-\psi(\zeta)=\log(1+a/\zeta)+O(|\zeta|^{-2})$. On
the compact set $\{(a,\zeta):a\in\mathcal A,\ \zeta\in\Omega_a,\
|\zeta|\le R\}$ the Gamma arguments stay away from their poles,
so $G_a$, $1/G_a$, and the two differences of digamma functions
are continuous. This gives the remaining bounds uniformly in
$a$.
\end{proof}

\section{Aharonov--Bohm weighted bounds}
\label{sec:interior}

\subsection{Reduction to a multiplier on the cylinder}
Fix $1<p<\infty$, $\beta\in\R$, and put
\begin{equation}\label{eq:c-weight}
  c=\frac{2+\beta}{p}.
\end{equation}
Let
\begin{equation}\label{eq:fulliso}
  (\mathcal U_{p,\beta}f)(u,\theta)=e^{cu}f(e^u,\theta),
\end{equation}
then it is an isometric isomorphism of $L^p_\beta$ onto $L^p(\R\times\T)$.
Let $\mathcal D$ be the space of functions $f=\sum_{|k|\le
N}e^{ik\theta}h_k(r)$ with $h_k\in C_c^\infty(0,\infty)$, which
is dense in $L^p_\beta\cap L^2(\R^2)$ for the norm
$\|\cdot\|_{L^p_\beta}+\|\cdot\|_2$. Indeed, one first truncates
to an annulus, then mollifies in the logarithmic radial variable
and uses angular Fej\'er means; all three operations converge
simultaneously in the two norms. 

For a bounded function $M$ on
$\R\times\Z$ we denote by $T_M$ the Fourier multiplier on
$L^2(\R\times\T)$, $\widehat{T_Mg}(\tau,k)=M(\tau,k)\widehat
g(\tau,k)$, where
\begin{equation*}
  \widehat g(\tau,k)=(2\pi)^{-1/2}
  \int_{\R\times\T}e^{-iu\tau-ik\theta}g(u,\theta)\dd
  u\dd\theta,
\end{equation*}
and we say that $M$ is an $L^p(\R\times\T)$ multiplier if $T_M$
is bounded on $L^p(\R\times\T)$.

Suppose for now that
\begin{equation}\label{eq:c-strip}
  -\rho_\alpha<c<2.
\end{equation}
In fact the common Mellin strip is exactly
\begin{equation}\label{eq:common-strip}
  \bigcap_{k\in\Z}S_{\nu_k,\mu_k}
  =\left(-\inf_{k\in\Z}\nu_k,\,2+\inf_{k\in\Z}\mu_k\right)
  =(-\rho_\alpha,2),
\end{equation}
because $\inf_k\nu_k=\rho_\alpha$ and $\inf_k\mu_k=0$. Thus
$\nu_k\ge\rho_\alpha$ and $\mu_k\ge0$ imply that $c\in
S_{\nu_k,\mu_k}$ for every $k$.
Proposition~\ref{prop:partial-wave}, the Mellin identity
\eqref{eq:symbol-id}, and $\widehat{e^{cu}h(e^u)}(\tau)=\Mell
h(c-i\tau)$ therefore give
\begin{equation}\label{eq:conjugation}
  \mathcal U_{p,\beta}W_\pm f=T_{M_{c,\pm}}\,\mathcal
  U_{p,\beta}f,
  \qquad f\in\mathcal D,
\end{equation}
where
\begin{equation}\label{eq:joint-mult}
  M_{c,\pm}(\tau,k)=e^{\mp i\delta_k}\,m_{\nu_k,\mu_k}(c-i\tau)
  =e^{\mp i\delta_k}\,
  \frac{\Gamma\bigl(\frac{\nu_k+c-i\tau}{2}\bigr)
    \Gamma\bigl(\frac{\mu_k+2-c+i\tau}{2}\bigr)}
  {\Gamma\bigl(\frac{\nu_k+2-c+i\tau}{2}\bigr)
    \Gamma\bigl(\frac{\mu_k+c-i\tau}{2}\bigr)}.
\end{equation}
Consequently, if $M_{c,\pm}$ is an $L^p(\R\times\T)$ multiplier,
then $W_\pm$ is bounded on $L^p_\beta$: the estimate follows
first on $\mathcal D$ from \eqref{eq:conjugation}. More
explicitly, if $f_n\in\mathcal D$ converges to $f\in
L^p_\beta\cap L^2$ in both norms, the multiplier estimate makes
$W_\pm f_n$ Cauchy in $L^p_\beta$, while $L^2$ unitarity gives
$W_\pm f_n\to W_\pm f$ in $L^2$. A subsequence converges almost
everywhere in both senses, so the two limits agree and the
weighted estimate holds for $W_\pm f$.

For each fixed $k$, the function $\tau\mapsto M_{c,\pm}(\tau,k)$
is smooth and bounded, and
\begin{equation}\label{eq:fixed-mode-mikhlin}
  |\partial_\tau M_{c,\pm}(\tau,k)|
  \le C_{c,k,\alpha}(1+|\tau|)^{-1}.
\end{equation}
Indeed, boundedness follows from Proposition~\ref{prop:symbol};
the derivative bound follows from the vertical Stirling
expansions for the Gamma and digamma functions when
$|\tau|\ge1$, and by compactness when $|\tau|\le1$. Thus every
fixed channel is an $L^p(\R)$ multiplier. These estimates alone
are not uniform enough to control the angular sum, so we next
establish joint regularity in $(\tau,k)$ by extending the large
mode symbol to $\R^2$.

\subsection{Two classical multiplier theorems}
We use the following two results.

\begin{theorem}[Marcinkiewicz;
{\cite[\S6.2]{Grafakos2014}}]\label{thm:marcinkiewicz}
Let $m\in L^\infty(\R^2)$ be of class $C^2$ on each open
quadrant. Suppose that, for $a,b\in\{0,1\}$,
\begin{equation}\label{eq:marcinkiewicz}
  |\xi_1|^a|\xi_2|^b
  |\partial_{\xi_1}^a\partial_{\xi_2}^b m(\xi_1,\xi_2)|\le B
\end{equation}
whenever $\xi_1\xi_2\ne0$. Then $m$ is an $L^p(\R^2)$ Fourier
multiplier for every $1<p<\infty$, with norm bounded by $C_pB$.
\end{theorem}

\begin{theorem}[de Leeuw; {\cite{deLeeuw1965},
\cite{Saeki1970}}]\label{thm:deleeuw}
Let $m:\R^2\to\C$ be bounded and continuous, and an $L^p(\R^2)$
Fourier multiplier for some $1<p<\infty$. Then its restriction
$(\tau,k)\mapsto m(\tau,k)$ to $\R\times\Z$ is an
$L^p(\R\times\T)$ multiplier, with norm not exceeding that of
$m$.
\end{theorem}

The mixed restriction used here is a standard closed subgroup
form of de~Leeuw's theorem; see \cite{deLeeuw1965} and
\cite[Corollary~4.6]{Saeki1970}. We also recall that the
indicator of a half line, $k\mapsto\ind_{\{k\ge K\}}$, is an
$L^p(\T)$ multiplier for $1<p<\infty$ (M.~Riesz's theorem, see
\cite[Ch.~4]{Grafakos2014}), and that multipliers on
$\R\times\T$ depending only on $\tau$, or only on $k$, are
bounded on $L^p(\R\times\T)$ as soon as they are bounded on
$L^p(\R)$, resp.\ $L^p(\T)$, by Fubini's theorem.

\subsection{Uniform extension of the symbol}
Let $K\in\N$ be an integer with $K\ge|\alpha|+|c|+3$. For
$\sigma\in\{+1,-1\}$ set
\begin{equation}\label{eq:aux}
  a_\sigma=\frac{\sigma\alpha}{2},\qquad
  d_\sigma=-\frac{\sigma\pi\alpha}{2},\qquad
  w_x=\frac{x+c-i\tau}{2},\qquad
  v_x=\frac{x+2-c+i\tau}{2},
\end{equation}
and, for $\tau\in\R$ and real $x\ge K-1$,
\begin{equation}\label{eq:F}
  F_{\sigma,\pm}(\tau,x)=e^{\mp
  id_\sigma}\,\frac{G_{a_\sigma}(w_x)}{G_{a_\sigma}(v_x)}.
\end{equation}
Since $x\ge K-1$ and $K\ge|\alpha|+|c|+3$, the real parts of
$w_x,w_x+a_\sigma,v_x,v_x+a_\sigma$ are all $\ge1$, so
Lemma~\ref{lem:gamma} applies with $\mathcal
A=\{-\alpha/2,\alpha/2\}$ and $\varepsilon=1$. For $k\ge K$ we
have $\mu_k=k$, $\nu_k=k+\alpha$, $\delta_k=-\pi\alpha/2=d_+$,
and for $k\le-K$ we have $\mu_k=-k$, $\nu_k=-k-\alpha$,
$\delta_k=\pi\alpha/2=d_-$; comparing with \eqref{eq:joint-mult}
we obtain the exact identities
\begin{equation}\label{eq:exact-tail}
  M_{c,\pm}(\tau,k)=F_{+,\pm}(\tau,k)\quad(k\ge K),\qquad
  M_{c,\pm}(\tau,k)=F_{-,\pm}(\tau,-k)\quad(k\le-K).
\end{equation}

\begin{lemma}\label{lem:mixed}
For $0\le a+b\le2$, $\tau\in\R$ and $x\ge K-1$,
\begin{equation}\label{eq:mixed-est}
  |\partial_\tau^a\partial_x^bF_{\sigma,\pm}(\tau,x)|
  \le C_{c,\alpha}\,(1+|\tau|+x)^{-a-b}.
\end{equation}
\end{lemma}

\begin{proof}
We have $|w_x|\simeq|v_x|\simeq1+|\tau|+x$, so \eqref{eq:G-size}
gives $|F_{\sigma,\pm}|\le C(|w_x|/|v_x|)^{a_\sigma}\le C$.
Locally, $F_{\sigma,\pm}=e^{\mp
id_\sigma}\exp(\Lambda(w_x)-\Lambda(v_x))$ with $\Lambda$ a
branch of $\log G_{a_\sigma}$, whose derivatives
$\Lambda'=\psi(\cdot+a_\sigma)-\psi$ and
$\Lambda''=\psi'(\cdot+a_\sigma)-\psi'$ are single valued. Since
$\partial_\tau w_x=-i/2$, $\partial_\tau v_x=i/2$,
$\partial_xw_x=\partial_xv_x=1/2$, every first derivative of
$F_{\sigma,\pm}$ is $F_{\sigma,\pm}$ times a linear combination
of $\Lambda'(w_x)$ and $\Lambda'(v_x)$, and every second
derivative is $F_{\sigma,\pm}$ times a linear combination of
$\Lambda''(w_x)$, $\Lambda''(v_x)$ and products of two first
order terms. Now \eqref{eq:G-deriv} gives the claim.
\end{proof}

Fix $\varrho\in C^\infty(\R)$ with $\varrho(\eta)=\eta$ for
$\eta\ge K$, $\varrho\ge K-1$ everywhere,
$\varrho(\eta)\simeq1+|\eta|$, and $\varrho'$ bounded (a
smoothing of $\eta\mapsto K+|\eta-K|$ on $[K-1,K]$ will do), and
define
\begin{equation}\label{eq:extension}
  Q_{+,\pm}(\tau,\eta)=F_{+,\pm}(\tau,\varrho(\eta)),\qquad
  Q_{-,\pm}(\tau,\eta)=F_{-,\pm}(\tau,\varrho(-\eta)),\qquad
  (\tau,\eta)\in\R^2.
\end{equation}
Then $Q_{\pm,\pm}\in C^\infty(\R^2)$, and since
$1+|\tau|+\varrho(\pm\eta)\simeq1+|\tau|+|\eta|$, the chain rule
and Lemma~\ref{lem:mixed} give, for $a,b\in\{0,1\}$,
\begin{equation}\label{eq:Q-est}
  |\partial_\tau^a\partial_\eta^b Q_{\sigma,\pm}(\tau,\eta)|
  \le C_{c,\alpha}(1+|\tau|+|\eta|)^{-a-b}.
\end{equation}
By \eqref{eq:exact-tail},
\begin{equation}\label{eq:Q-tail}
  Q_{+,\pm}(\tau,k)=M_{c,\pm}(\tau,k)\ \ (k\ge K),\qquad
  Q_{-,\pm}(\tau,k)=M_{c,\pm}(\tau,k)\ \ (k\le-K).
\end{equation}

\subsection{The weighted multiplier bound}

\begin{proposition}\label{prop:weighted-suff}
Let $\alpha\in\R$, $1<p<\infty$, and $\beta\in\R$. If
$c=(2+\beta)/p$ satisfies $-\rho_\alpha<c<2$, then $W_\pm$ is
bounded on $L^p_\beta$.
\end{proposition}

\begin{proof}
By \eqref{eq:conjugation} it suffices to show that $M_{c,\pm}$
is an $L^p(\R\times\T)$ multiplier. Decompose
\begin{equation}\label{eq:decomp}
  M_{c,\pm}(\tau,k)=\ind_{\{k\ge K\}}Q_{+,\pm}(\tau,k)
  +\ind_{\{k\le-K\}}Q_{-,\pm}(\tau,k)
  +\sum_{|j|<K}\ind_{\{k=j\}}\,M_{c,\pm}(\tau,j),
\end{equation}
which is an identity by \eqref{eq:Q-tail}. Since
$|\tau|^a|\eta|^b\le(1+|\tau|+|\eta|)^{a+b}$, \eqref{eq:Q-est}
implies \eqref{eq:marcinkiewicz}.
Theorem~\ref{thm:marcinkiewicz} therefore shows that
$Q_{\sigma,\pm}$ is an $L^p(\R^2)$ multiplier; it is bounded and
continuous, so by Theorem~\ref{thm:deleeuw} its restriction to
$\R\times\Z$ is an $L^p(\R\times\T)$ multiplier. The factors
$\ind_{\{k\ge K\}}$ and $\ind_{\{k\le-K\}}$ are $L^p(\T)$
multipliers by M.~Riesz's theorem, and the product of two $L^p$
multipliers is an $L^p$ multiplier. In the last sum,
$\ind_{\{k=j\}}$ is the projection $P_j$, a contraction on
$L^p(\R\times\T)$, and $\tau\mapsto M_{c,\pm}(\tau,j)$ is an
$L^p(\R)$ multiplier by \eqref{eq:fixed-mode-mikhlin} and the
one dimensional Mikhlin theorem. Hence each of the finitely many
terms in \eqref{eq:decomp} is an $L^p(\R\times\T)$ multiplier,
and so is $M_{c,\pm}$. 
\end{proof}

\begin{proof}[Proof of
Corollary~\ref{cor:AB-unweighted}]
For $\beta=0$ one has $c=2/p\in(0,2)$, so
Proposition~\ref{prop:weighted-suff} gives the bound for
$W_\pm$. The bound for $W_\pm^*$ follows by unweighted duality.
For local uniformity, let $E\subset\R$ be compact, put
$A_0=\sup_{\alpha\in E}|\alpha|$, and choose one integer $K\ge
A_0+2/p+3$ for every $\alpha\in E$. The large mode estimates
then use the single compact parameter set $[-A_0/2,A_0/2]$ in
Lemma~\ref{lem:gamma}; the same smoothing function $\varrho$ can
be used for all $\alpha\in E$. For the finitely many modes
$|j|<K$, the line $\Re z=2/p$ stays a fixed positive distance
from both boundaries of $S_{\nu_j,\mu_j}$. On $|\tau|\le1$, the
symbols and their first derivatives are therefore bounded
uniformly by compactness in $(\alpha,\tau)$; on $|\tau|\ge1$,
the vertical Gamma and digamma expansions are uniform because
all real parts range over a fixed compact interval. Thus the one
dimensional Mikhlin constants of all the remaining modes are
uniform on $E$.
For $\alpha=0$, the partial wave formula gives the identity
in every channel.
\end{proof}

\begin{remark}
The extension \eqref{eq:extension} is the only place where the
specific form of the Aharonov--Bohm symbol is used beyond the
Gamma ratio estimates: for $|k|\ge K$ the symbol depends on $k$
through the two real analytic functions $x\mapsto
F_{\pm,\pm}(\tau,x)$, and it is this analytic dependence on the
mode index, together with the decay of the derivatives in
$(\tau,x)$ jointly, that yields uniformity in $k$.
\end{remark}

\section{General electromagnetic bounds}
\label{sec:electromagnetic}

We now prove Theorem~\ref{thm:general}. The scalar scattering
formula is already available from
Proposition~\ref{prop:em-L2}, and the magnetic bound from
Corollary~\ref{cor:AB-unweighted}. We first treat
$H_{\alpha,a}$, defined in \eqref{eq:em-H}, and then use a
smooth angular gauge to handle $\mathbf A=b e_\theta$.

\begin{proposition}[Constant magnetic coefficient]
\label{prop:em-wave}
  Let $\alpha\in\R\setminus\frac12\Z$ and let
  $a\in W^{3,\infty}(\T;\R)$ satisfy $B_{\alpha,a}\ge0$.
  The wave operators
  \begin{equation}\label{eq:em-wave}
    W_\pm^{\alpha,a}
    =\slim_{t\to\pm\infty}e^{itH_{\alpha,a}}e^{it\Delta}
  \end{equation}
  exist and are unitary on $L^2(\R^2)$. For every
  $1<p<\infty$, both $W_\pm^{\alpha,a}$ and their
  adjoints extend boundedly to $L^p(\R^2)$ and are mutually
  inverse there. The same assertions hold for the wave
  operators of the pair $(H_{\alpha,a},H_\alpha)$.
\end{proposition}

\subsection{Angular changes of basis and Fourier conjugation}

Let $\mathcal F$ be the unitary Fourier transform on $\R^2$,
with kernel $(2\pi)^{-1}e^{-ix\cdot\xi}$. Set
$\mathcal F_+=\mathcal F$ and
$\mathcal F_-=\mathcal F^{-1}$. Angular operators act at each
fixed radius. We use $M_h$ for multiplication by $h(\theta)$
and $R_\varphi f(r,\theta)=f(r,\theta-\varphi)$ for rotations.

\begin{lemma}\label{lem:em-angular}
  Let $e_k=(2\pi)^{-1/2}e^{ik\theta}$, and suppose that a
  unitary angular operator $J$ satisfies
  $Je_k=e_k(1+r_k)$, where
  \begin{equation}\label{eq:em-summability}
    \sum_{k\in\Z}\|r_k\|_{H^3(\T)}^2<\infty.
  \end{equation}
  Then $J,J^*$ and
  $\mathcal F_\pm^{-1}J\mathcal F_\pm$,
  $\mathcal F_\pm^{-1}J^*\mathcal F_\pm$ are bounded on
  $L^p(\R^2)$ for every $1<p<\infty$.
\end{lemma}

\begin{proof}
  Define
  \begin{equation*}
    K(\theta,\varphi)
    =\sum_{k\in\Z}r_k(\theta)e^{ik\varphi},
  \end{equation*}
  with convergence in $L^2(\T_\varphi;H^3(\T_\theta))$.
  Parseval, the embedding $H^3(\T)\subset C^2(\T)$, and
  Cauchy--Schwarz imply
  \begin{equation}\label{eq:em-kernel-bound}
    \int_0^{2\pi}\|K(\cdot,\varphi)\|_{C^2}
    \frac{d\varphi}{2\pi}
    \le C\left(\sum_k\|r_k\|_{H^3}^2\right)^{1/2}.
  \end{equation}
  On finite angular sums, and then on $L^2$, one has
  \begin{equation}\label{eq:em-rotation}
    J-I=\int_0^{2\pi}
    M_{K(\cdot,\varphi)}R_\varphi
    \frac{d\varphi}{2\pi}.
  \end{equation}
  Indeed, applying the integral to $e_k$ selects the Fourier
  coefficient $r_k$. Rotations are $L^p$ isometries, so
  \eqref{eq:em-kernel-bound} already proves boundedness of
  $J$ on $L^p$.

  Rotations commute with $\mathcal F_\pm$. Moreover,
  $\mathcal F_\pm^{-1}M_h\mathcal F_\pm$ is a homogeneous
  Fourier multiplier with angular symbol $h(\theta)$ or
  $h(\theta+\pi)$. The two dimensional Mikhlin theorem gives
  \begin{equation*}
    \|\mathcal F_\pm^{-1}M_h\mathcal F_\pm\|_{L^p\to L^p}
    \le C_p\|h\|_{C^2},\qquad 1<p<\infty;
  \end{equation*}
  see \cite{Grafakos2014}. Conjugating
  \eqref{eq:em-rotation} on $L^2$ and integrating these bounds
  proves the corresponding $L^p$ estimate by density. The
  integrals can be understood as Bochner integrals applied to
  a function in $L^2\cap L^p$.

  The adjoint of the right hand side of
  \eqref{eq:em-rotation} has the same form, with kernel
  \begin{equation*}
    K^*(\theta,\varphi)
    =\overline{K(\theta-\varphi,-\varphi)}.
  \end{equation*}
  Its integrated $C^2$ norm equals that in
  \eqref{eq:em-kernel-bound}. This proves both assertions for
  $J^*$. The inverse identities on $L^2$ extend to $L^p$ by
  density.
\end{proof}

\begin{lemma}\label{lem:em-eigenbasis}
  Under the assumptions of Proposition~\ref{prop:em-wave}, there
  is an orthonormal eigenbasis $\{\phi_k\}_{k\in\Z}$ of
  $B_{\alpha,a}$ such that, writing
  $\phi_k=e_k b_k$ and $B_{\alpha,a}\phi_k=\lambda_k\phi_k$,
  \begin{equation}\label{eq:em-basis-estimate}
    |\lambda_k-(k+\alpha)^2|\le\|a\|_\infty,
    \qquad
    \|b_k-1\|_{H^3(\T)}
    \le\frac{C_{\alpha,a}}{1+|k|}.
  \end{equation}
  Consequently, $Je_k=\phi_k$ satisfies
  Lemma~\ref{lem:em-angular}.
\end{lemma}

\begin{proof}
  The eigenvalues of $B_{\alpha,0}$ are $(k+\alpha)^2$
  and are distinct since $2 \alpha\not\in \mathbb{Z}$.
  Arrange them in a strictly increasing sequence $\mu_{n}$,
  and arrange the $\lambda_{k}$ in another increasing sequence
  $\mu'_{n}$ (with multiplicity).
  The min--max variational characterization implies immediately
  $|\mu_{n}-\mu'_{n}|\le\|a\|_{L^{\infty}}$.
  Since $|\mu_{n}-\mu_{n+1}|\to \infty$, the intervals
  $[\mu_{n}-\|a\|_{L^{\infty}},\mu_{n}+\|a\|_{L^{\infty}}]$
  are disjoint for large $n$, and we can match each $\mu'_{n}$ 
  uniquely
  with one $\mu_{n}$; in particular, the eigenvalues $\lambda_{k}$
  of $B_{\alpha,a}$ are simple for large $k$, and, after
  a suitable reordering, satisfy the first of 
  \eqref{eq:em-basis-estimate}.
  The eigenfunctions solve the ODE
  $$-\phi_k''-2i\alpha\phi_k' +\bigl(\alpha^2+a(\theta)\bigr)\phi_k
    =\lambda_k\phi_k$$
  hence $\phi_{k}\in H^{3}$ by bootstrapping.

  Put $z_k=k+\alpha$, let
  $Qf=f-\frac1{2\pi}\int_0^{2\pi}f(\theta)\,d\theta$
  be the projection on zero angular mean functions,
  and let $A_k=D_\theta(D_\theta+2z_k)$ where 
  $D_{\theta}=-i \partial_{\theta}$.
  On functions of mean zero, the operator
  $A_k$ is invertible. To estimate
  its inverse, set $\delta=\operatorname{dist}(2\alpha,\Z)>0$.
  For each nonzero integer $\ell$, if $|\ell|<|z_k|$, then
  $|\ell+2z_k|>|z_k|$. If $|\ell|\ge|z_k|$, then
  $|\ell+2z_k|\ge\delta$. In either case,
  \begin{equation*}
    |\ell(\ell+2z_k)|\ge\min\{1,\delta\}|z_k|.
  \end{equation*}
  Since $A_k$ is a Fourier multiplier, this proves, for
  sufficiently large $|k|$,
  \begin{equation}\label{eq:em-inverse-gap}
    \|A_k^{-1}Q\|_{H^3\to H^3}
    \le\frac{C_\alpha}{1+|k|}.
  \end{equation}
  Write $b_k=t_k+v_k$, where $t_k$ is constant and $Qv_k=v_k$,
  and put $d_k=\lambda_k-z_k^2$. The eigenvalue equation
  becomes
  \begin{equation*}
    A_kv_k=Q\bigl((d_k-a)v_k\bigr)-t_kQa.
  \end{equation*}
  Multiplication by $a$ is bounded on $H^3$, while $d_k$ is
  uniformly bounded. Apply \eqref{eq:em-inverse-gap} and
  absorb the $v_k$ term for large $|k|$ to obtain
  \begin{equation*}
    \|v_k\|_{H^3}\le\frac{C_{\alpha,a}|t_k|}{1+|k|}.
  \end{equation*}
  In particular, $t_k$ cannot vanish for a normalized
  eigenfunction. Choose its phase so that $t_k>0$. With the
  normalized angular measure $d\theta/(2\pi)$, normalization
  gives $t_k^2+\|v_k\|_2^2=1$, whence $t_k\le1$ and
  $|t_k-1|=O(|k|^{-2})$. This proves the second estimate for
  large $|k|$. Choose a smooth orthonormal eigenbasis for the
  remaining finite dimensional subspace and increase the
  constant to cover it. Summing $(1+|k|)^{-2}$ proves
  \eqref{eq:em-summability}.
\end{proof}

\subsection{From spectral transplantation to wave operators}

\begin{proposition}[Spectral transplantation]
\label{prop:em-transplantation}
  Let $a$ and $\alpha$ be as in
  Proposition~\ref{prop:em-wave}. Choose the eigenbasis of
  Lemma~\ref{lem:em-eigenbasis} and set
  \begin{equation*}
    \sigma_k=|k+\alpha|,\qquad \tau_k=\sqrt{\lambda_k},
    \qquad Je_k=\phi_k.
  \end{equation*}
  The operator
  \begin{equation}\label{eq:em-S}
    S\left(\sum_k f_ke_k\right)
    =\sum_k(\HH_{\tau_k}\HH_{\sigma_k}f_k)\phi_k
  \end{equation}
  is unitary on $L^2(\R^2)$. Both $S$ and $S^*$ extend
  boundedly to every $L^p(\R^2)$, $1<p<\infty$, and are
  mutually inverse there. For every bounded Borel function
  $m$ on $[0,\infty)$,
  \begin{equation*}
    m(H_{\alpha,a})S=Sm(H_\alpha)\qquad\text{on }L^2.
  \end{equation*}
\end{proposition}

The $L^p$ assertion is a special case of
\cite[Theorem~2.13]{FSWZZ2026}. In order to make the paper
self contained, we include a proof under the hypotheses above 
in Appendix~\ref{app:transplantation}.

\begin{proof}[Proof of Proposition~\ref{prop:em-wave}]
  Use $\sigma_k,\tau_k,J,S$ from
  Proposition~\ref{prop:em-transplantation}.
  To simplify notation within this proof,
  following the general construction~\eqref{eq:em-K}, write
  \begin{equation*}
    \mathcal K_a=\mathcal K_{B_{\alpha,a}},\qquad
    \mathcal K_\alpha=\mathcal K_{B_{\alpha,0}},\qquad
    \mathcal K_0=\mathcal K_{D_\theta^2}.
  \end{equation*}
  More explicitly, if $\phi_k$ is the eigenbasis from
  Lemma~\ref{lem:em-eigenbasis}, then
  \begin{equation*}
    \begin{aligned}
      \mathcal K_a\Bigl(\sum_k f_k\phi_k\Bigr)
        =\sum_k(\HH_{\tau_k}f_k)\phi_k,\quad & \quad
      \mathcal K_\alpha\Bigl(\sum_k f_ke_k\Bigr)
        =\sum_k(\HH_{\sigma_k}f_k)e_k,\\
      \mathcal K_0\Bigl(\sum_k f_ke_k\Bigr)
        &=\sum_k(\HH_{|k|}f_k)e_k.
    \end{aligned}
  \end{equation*}
  Here the subscript $0$ denotes the free angular operator
  $D_\theta^2$, in parallel with $H_0=-\Delta$.
  All identities between these operators are initially on
  $L^2$. Formula~\eqref{eq:em-S} gives
  \begin{equation*}
    S=\mathcal K_aJ\mathcal K_\alpha.
  \end{equation*}
  Proposition~\ref{prop:em-transplantation} supplies the
  $L^p$ bounds for $S$ and $S^*$.

  Next put $V_\alpha=\mathcal K_\alpha\mathcal K_0$.
  Corollary~\ref{cor:AB-unweighted} implies that
  $V_\alpha,V_\alpha^*$ are bounded on every $L^p$ in
  question. Indeed, the partial wave formula differs from
  $V_\alpha$ by the angular phases
  $e^{\pm i\pi(\sigma_k-|k|)/2}$. Each phase sequence is
  constant on each sufficiently large positive or negative
  tail. The corresponding multiplier is a linear combination
  of the angular Riesz projections and finitely many angular
  projections, all bounded on $L^p$. It follows that
  \begin{equation}\label{eq:em-U}
    U=SV_\alpha=\mathcal K_aJ\mathcal K_0
  \end{equation}
  and $U^{-1}=U^*$ are bounded on $L^p$.

  Define the diagonal angular multipliers
  \begin{equation}\label{eq:em-D}
    D_\pm e_k
    =e^{\pm i\pi(\tau_k-|k|)/2}e_k.
  \end{equation}
  By \eqref{eq:em-basis-estimate},
  $\tau_k-\sigma_k=O(|k|^{-1})$. Consequently, the sequence
  in \eqref{eq:em-D} equals
  $e^{\pm i\pi\alpha/2}$ on the positive tail and
  $e^{\mp i\pi\alpha/2}$ on the negative tail, up to an
  $\ell^2(\Z)$ remainder. The constant tails give angular
  Riesz projections, and the remainder again has an $L^1$
  convolution kernel. This proves boundedness of $D_\pm$
  and $D_\pm^{-1}$ on $L^p$.

  The Fourier transform in polar coordinates can be written
  for $f=\sum_k f_ke_k$
  \begin{equation}\label{eq:em-free-Fourier}
    \mathcal F_\pm\Bigl(\sum_k f_ke_k\Bigr)
        =\sum_k e^{\mp i\pi|k|/2}
          (\HH_{|k|}f_k)e_k =
      \mathcal F_\pm
        =e^{\mp i\pi|D_\theta|/2}\mathcal K_0.
  \end{equation}
  Proposition~\ref{prop:em-L2}, applied with $B_0=D_\theta^2$,
  now yields the exact factorization
  \begin{equation}\label{eq:em-factorization}
    W_\pm^{\alpha,a}
    =U D_\pm\mathcal F_\pm^{-1}J^*\mathcal F_\pm.
  \end{equation}
  For clarity, define $T_\pm e_k=e^{\pm i\pi\tau_k/2}e_k$.
  The diagonal operators commute with $\mathcal K_0$, and
  \eqref{eq:em-free-Fourier} gives
  \begin{equation*}
    U D_\pm\mathcal F_\pm^{-1}J^*\mathcal F_\pm
      =\mathcal K_aJ T_\pm J^*\mathcal F_\pm =
      \mathcal K_a e^{\pm i\pi\sqrt{B_{\alpha,a}}/2}
        \mathcal F_\pm.
  \end{equation*}
  The last expression is exactly \eqref{eq:em-L2} for the
  free comparison operator. This also checks both signs of
  the scattering phases.

  Lemmas~\ref{lem:em-angular} and \ref{lem:em-eigenbasis}
  control the Fourier conjugate of $J^*$ and its inverse.
  Every factor in \eqref{eq:em-factorization} is therefore
  bounded and invertible on $L^p$. Proposition~\ref{prop:em-L2}
  gives existence and unitarity on $L^2$, and density extends
  the inverse identities to $L^p$. Finally, the same
  proposition, or the chain rule for these unitary wave
  operators, gives
  \begin{equation*}
    W_\pm(H_{\alpha,a},H_\alpha)
    =W_\pm^{\alpha,a}(W_\pm^\alpha)^*,
    \qquad W_\pm^\alpha=W_\pm(H_\alpha,-\Delta).
  \end{equation*}
  Corollary~\ref{cor:AB-unweighted} proves the asserted
  bounds for this pair as well.
\end{proof}

\subsection{The general magnetic coefficient}

\begin{proof}[Proof of Theorem~\ref{thm:general}]
  The function
  \begin{equation*}
    g(\theta)=\exp\left(i\int_0^\theta
    (\alpha-b(s))\,ds\right)
  \end{equation*}
  is periodic and has modulus one. Direct differentiation
  gives $(D_\theta+b)M_g=M_g(D_\theta+\alpha)$. This identity
  on the punctured minimal domain extends to the Friedrichs
  quadratic forms, giving
  $\mathcal L_{\mathbf A,a}=M_gH_{\alpha,a}M_{\bar g}$.
  In particular, the angular nonnegativity assumptions are
  equivalent. Proposition~\ref{prop:em-L2} gives
  \begin{equation}\label{eq:em-gauge-wave}
    W_\pm(\mathcal L_{\mathbf A,a},-\Delta)
    =M_gW_\pm^{\alpha,a}
    \mathcal F_\pm^{-1}M_{\bar g}\mathcal F_\pm.
  \end{equation}
  Indeed, both the angular functional calculus and
  $\mathcal K_B$ conjugate by $M_g$, so their product in
  \eqref{eq:em-L2} has exactly this form.
  Multiplication by $g$ is an $L^p$ isometry, while the last
  factor in \eqref{eq:em-gauge-wave} and its inverse are
  bounded homogeneous Fourier multipliers by Mikhlin's
  theorem. Proposition~\ref{prop:em-wave} proves the free
  comparison.
  Both operators of the relative pair conjugate by the same
  $M_g$, so
  \begin{equation*}
    W_\pm(\mathcal L_{\mathbf A,a},\mathcal L_{\mathbf A,0})
    =M_g W_\pm(H_{\alpha,a},H_\alpha)M_{\bar g}.
  \end{equation*}
  The relative assertion follows from the same proposition.
\end{proof}

\begin{remark}[Regularity and exceptional flux]
\label{rem:em-scope}
  The $L^2$ formula in Proposition~\ref{prop:em-L2} does not
  require smoothness or the exclusion of $\frac12\Z$.
  These restrictions enter the present $L^p$ proof through
  the angular eigenbasis estimates. Reaching
  $a,b\in W^{1,\infty}$ would require replacing the $H^3$
  argument by a multiplier estimate with less angular
  regularity. Flux in $\frac12\Z$ requires treating the
  degenerate angular clusters rather than inverting
  $D_\theta(D_\theta+2(k+\alpha))$ on all functions of mean
  zero. Neither extension is asserted here, and no weighted
  or endpoint conclusion for general coefficients is
  included. The substantive extra step beyond composing
  spectral intertwiners is the scattering normalization
  \eqref{eq:em-factorization}, including the bounded Fourier
  conjugate of the angular change of basis.
\end{remark}

\begin{remark}[Spectral intertwining at half flux]
\label{rem:half-flux-intertwining}
  Suppose that $a,b\in W^{1,\infty}(\T;\R)$, $\alpha=1/2$,
  $B_{\mathbf A,a}\ge0$, and
  $a(\pi-\theta)=a(\pi+\theta)$ for $0\le\theta\le\pi$.
  By \cite[Theorem~2.13]{FSWZZ2026}, there is a unitary
  spectral intertwiner $S$ from $\mathcal L_{\mathbf A,0}$
  to $\mathcal L_{\mathbf A,a}$ such that $S$ and $S^*$
  are bounded on every $L^p$, $1<p<\infty$.
  With $g$ as in the preceding proof, set
  \begin{equation*}
    U=S M_g W_+(H_{1/2},-\Delta).
  \end{equation*}
  The gauge identity and Corollary~\ref{cor:AB-unweighted}
  show that $U$ is unitary on $L^2$, while $U$ and $U^*$
  extend to bounded, mutually inverse operators on every
  $L^p$, $1<p<\infty$. For every bounded Borel function $m$,
  \begin{equation*}
    m(\mathcal L_{\mathbf A,a})
    =U m(-\Delta)U^*\qquad\text{on }L^2.
  \end{equation*}
  Thus spectral similarity with the Laplacian, and the
  resulting $L^p$ transference, remain available at $\alpha=1/2$
  under this symmetry assumption. This composition
  does not however identify $U$ with the wave operator
  of the pair $(\mathcal L_{\mathbf A,a},-\Delta)$.
\end{remark}

\section{Sharpness and the endpoints
\texorpdfstring{$p=1$ and $p=\infty$}{p=1 and p=infty}}
\label{sec:endpoints}

We begin by collecting the features of the symbol 
$m_{\nu,\mu}(z)$ defined in \eqref{eq:symbol} that obstruct
the endpoint bounds, together with the small radius asymptotic
needed for the sharp weighted range.

\begin{lemma}[Limits, poles, and radial asymptotics]
\label{lem:sing}
Let $\mu\ge0$ and let $\nu>-1$ be real. Set
$\phi=\frac\pi2(\nu-\mu)$.
\begin{enumerate}
\item[(i)] For every $c\in(-\nu,\mu+2)$,
           $\displaystyle\lim_{\tau\to\pm\infty}
           m_{\nu,\mu}(c-i\tau)
           =e^{\mp i\phi}$.
\item[(ii)] If $\mu=0$ and $\nu\ne0$, then $m_{\nu,0}$ extends
            meromorphically to $-\nu<\Re z<4$ with a single,
            simple pole at $z=2$, of residue $-\nu$, and for
            every $h\in C_c^\infty(0,\infty)$
\begin{equation}\label{eq:tail}
  (\HH_\nu\HH_0h)(r)=\nu\,r^{-2}\int_0^\infty h(s)\,s\dd
  s+O(r^{-3}),
  \qquad r\to\infty.
\end{equation}
\item[(iii)] Suppose that $(\mu-\nu)/2\notin\{0,-1,-2,\ldots\}$.
             Then, for every $h\in C_c^\infty(0,\infty)$,
\begin{equation}\label{eq:origin-asymptotic}
  (\HH_\nu\HH_\mu h)(r)
  =b_{\nu,\mu}\,r^\nu\int_0^\infty h(s)s^{-\nu-1}\dd s
  +O(r^{\nu+1}),
  \qquad r\to0,
\end{equation}
where
\begin{equation}\label{eq:origin-coefficient}
  b_{\nu,\mu}
  =\frac{2\Gamma\bigl(\frac{\mu+\nu+2}{2}\bigr)}
  {\Gamma(\nu+1)\Gamma\bigl(\frac{\mu-\nu}{2}\bigr)}\ne0.
\end{equation}
\end{enumerate}
\end{lemma}

\begin{proof}
(i) In \eqref{eq:symbol-G} with $z=c-i\tau$ we have
$\zeta_1=\frac{\mu+c-i\tau}2$, $\zeta_2=\frac{\mu+2-c+i\tau}2$,
so $|\zeta_1|/|\zeta_2|\to1$, $\arg\zeta_1\to\mp\pi/2$ and
$\arg\zeta_2\to\pm\pi/2$ as $\tau\to\pm\infty$. The general
Gamma ratio asymptotic \cite[Eq.~5.11.12]{NIST}, applied in
sectors avoiding the negative real axis, gives
$$m_{\nu,\mu}(c-i\tau)=(|\zeta_1|/|\zeta_2|)^a
e^{ia(\arg\zeta_1-\arg\zeta_2)} (1+O(|\tau|^{-1}))\to e^{\mp
i\pi a}=e^{\mp i\phi}.$$

(ii) For $\mu=0$,
$m_{\nu,0}(z)=\Gamma(\frac{\nu+z}2)\Gamma(\frac{2-z}2)
/\bigl(\Gamma(\frac{\nu+2-z}2)\Gamma(\frac z2)\bigr)$. On
$-\nu<\Re z<4$ the only singularity is the simple pole of
$\Gamma(\frac{2-z}2)$ at $z=2$; since
$\Gamma(\frac{2-z}2)=-\frac2{z-2}+O(1)$ and the other factors
equal $\Gamma(\frac\nu2+1)/\Gamma(\frac\nu2)=\frac\nu2$ at
$z=2$, the residue is $-\nu$. Let $h\in C_c^\infty(0,\infty)$,
$f=\HH_\nu\HH_0h$ and $F=m_{\nu,0}\Mell h$. By
Proposition~\ref{prop:symbol} and its proof, $f(r)=\frac1{2\pi
i}\int_{\Re z=1}F(z)r^{-z}dz$, and $F$ is rapidly decreasing on
vertical lines, uniformly on $\{1\le\Re z\le3,\
|z-2|\ge\frac12\}$: indeed $\Mell h$ is, and $m_{\nu,0}$ is
bounded there, because by Stirling's formula
\cite[Eq.~5.11.9]{NIST}
$$|\Gamma(x+iy)|=\sqrt{2\pi}\,|y|^{x-1/2}e^{-\pi|y|/2}(1+o(1))$$
as $|y|\to\infty$, uniformly for $x$ in compact sets, and the
exponents $x-\frac12$ of the four Gamma factors in $m_{\nu,0}$
add up to zero. Shifting the contour to $\Re z=3$ across the
pole at $z=2$ gives
\begin{equation*}
  f(r)=-\Res_{z=2}\bigl(F(z)r^{-z}\bigr)+\frac1{2\pi
  i}\int_{\Re z=3}F(z)r^{-z}dz
  =\nu\,\Mell h(2)\,r^{-2}+O(r^{-3}),
\end{equation*}
which is \eqref{eq:tail}, since $\Mell h(2)=\int_0^\infty
h(s)s\,ds$. Here the horizontal sides of the contour tend to
zero by the rapid vertical decay of $F$. The minus sign in front
of the residue is the sign obtained when the inverse Mellin
contour is moved to the right. The last integral is $O(r^{-3})$
because $F(3+i\cdot)\in L^1(\R)$.

(iii) The meromorphic continuation of \eqref{eq:symbol} has a
simple pole at $z=-\nu$, unless it is cancelled by a pole of
$\Gamma((\mu+z)/2)$ in the denominator. Under the stated
assumption there is no cancellation, and
\begin{equation*}
  \Res_{z=-\nu}m_{\nu,\mu}(z)
  =\frac{2\Gamma\bigl(\frac{\mu+\nu+2}{2}\bigr)}
  {\Gamma(\nu+1)\Gamma\bigl(\frac{\mu-\nu}{2}\bigr)}
  =b_{\nu,\mu}.
\end{equation*}
The next possible pole of $\Gamma((\nu+z)/2)$ is at $z=-\nu-2$,
while the reciprocal Gamma factors can only contribute zeros.
Hence, starting from the inverse Mellin representation on any
line in $S_{\nu,\mu}$ and shifting the contour to $\Re z=-\nu-1$
crosses only the pole at $z=-\nu$. As above, the horizontal
integrals vanish. A shift to the left contributes the residue
with a plus sign, and therefore
\begin{equation*}
  f(r)=b_{\nu,\mu}\Mell h(-\nu)r^\nu
  +\frac1{2\pi i}\int_{\Re z=-\nu-1}m_{\nu,\mu}(z)\Mell
  h(z)r^{-z}\dd z.
\end{equation*}
The vertical Gamma estimates make the integrand, apart from
$r^{-z}$, integrable on this line. The last integral is
consequently $O(r^{\nu+1})$, while $\Mell h(-\nu)=\int_0^\infty
h(s)s^{-\nu-1}\dd s$; this is \eqref{eq:origin-asymptotic}.
\end{proof}

\begin{lemma}\label{lem:L1-mult}
Let $S$ be a bounded operator on $L^1(\R)$ such that, for $g\in
C_c^\infty(\R)$, $\widehat{Sg}=\sigma\,\widehat g$ for a
continuous function $\sigma$ on $\R$. Then $\sigma$ is the
Fourier--Stieltjes transform of a finite complex Borel measure
on $\R$, and if the limits $\lim_{\tau\to+\infty}\sigma(\tau)$
and $\lim_{\tau\to-\infty}\sigma(\tau)$ both exist, they are
equal.
\end{lemma}

\begin{proof}
$S$ commutes with translations (this is true on
$C_c^\infty(\R)$, since $S$ is a Fourier multiplier there, and
extends by density), so by a classical theorem
\cite[Ch.~I, Thm.~3.19]{SteinWeiss1971} there is a finite
Borel measure
$\lambda$ with $Sg=\lambda*g$; then $\widehat\lambda\,\widehat g
=\sigma\widehat g$ for all $g\in C_c^\infty$, whence
$\sigma=\widehat\lambda$. Next,
\begin{equation*}
  \frac1T\int_0^T\bigl(\widehat\lambda(\tau)
    -\widehat\lambda(-\tau)\bigr)d\tau
  =-2i\int_\R\frac{1-\cos(Tx)}{Tx}\,d\lambda(x)
    \xrightarrow[T\to\infty]{}0
\end{equation*}
by dominated convergence, since $|(1-\cos y)/y|\le1$ and
$(1-\cos(Tx))/(Tx)\to0$ for every $x\neq0$. If $\sigma(\tau)\to
L_\pm$ as $\tau\to\pm\infty$, the left hand side tends to
$L_+-L_-$, so $L_+=L_-$.
\end{proof}

We can now prove the strong endpoint failures in
Theorem~\ref{thm:weighted}.

\begin{proposition}\label{prop:endpoints}
Let $\alpha\ne0$. Then neither $W_\pm$ nor $W_\pm^*$ is bounded
on $L^1(\R^2)$ or on $L^\infty(\R^2)$.
\end{proposition}

\begin{proof}
\emph{Step 0: reduction to a mode.} Suppose $W_\pm$ is bounded
on $L^p(\R^2)$, $p\in\{1,\infty\}$, with constant $C$. Since
$P_k$ commutes with $W_\pm$ and is a contraction on $L^p$, for
$f=e^{ik\theta}h(r)$ with $h\in L^2\cap L^p(\R_+,r\,dr)$ we get
$W_\pm f=e^{ik\theta}(W_{\pm,k}h)(r)$ and
$\|W_{\pm,k}h\|_{L^p(r\,dr)}\le C\|h\|_{L^p(r\,dr)}$: each
$W_{\pm,k}=e^{\mp i\delta_k}\HH_{\nu_k}\HH_{\mu_k}$ is bounded
on $L^p(\R_+,r\,dr)$. Likewise, if $W_\pm$ is bounded on
$L^\infty$, then $W_\pm^*$ is bounded on $L^1$: for $f\in
L^1\cap L^2$,
\begin{equation*}
  \|W_\pm^*f\|_1=\sup_{g}|\langle W_\pm^*f,g\rangle|
  =\sup_{g}|\langle f,W_\pm g\rangle|\le C\|f\|_1,
\end{equation*}
the supremum being taken over simple functions $g$ with bounded
support and $\|g\|_\infty\le1$; then each $(W_{\pm,k})^*=e^{\pm
i\delta_k}\HH_{\mu_k}\HH_{\nu_k}$ is bounded on
$L^1(\R_+,r\,dr)$.

\emph{Step 1: the $|x|^{-2}$ tail.} Let $\nu>0$. Then
$\HH_\nu\HH_0$ is not bounded on $L^1(\R_+,r\,dr)$: taking $h\in
C_c^\infty(0,\infty)$, $h\ge0$, $h\not\equiv0$, we have $h\in
L^1\cap L^2$, while by \eqref{eq:tail}
$|\HH_\nu\HH_0h(r)|\ge\frac\nu2 r^{-2}\int h\,s\,ds$ for large
$r$, which is not integrable with respect to $r\,dr$.

\emph{Step 2: $L^1$.} In the mode $k=0$ we have $\mu_0=0$ and
$\nu_0=|\alpha|>0$, so $W_{\pm,0}=e^{\mp
i\delta_0}\HH_{|\alpha|}\HH_0$ is unbounded on $L^1(r\,dr)$ by
Step~1, and $W_\pm$ is unbounded on $L^1(\R^2)$ by Step~0.
Explicitly, for radial $f(x)=h(|x|)$ as in Step~1,
\begin{equation*}
  W_\pm f(x)=e^{\mp
  i\delta_0}\Bigl(|\alpha|\,|x|^{-2}\int_0^\infty
  h(s)\,s\,ds+O(|x|^{-3})\Bigr),
  \qquad |x|\to\infty.
\end{equation*}

\emph{Step 3: $L^\infty$, integer flux.} Let
$\alpha\in\Z\setminus\{0\}$ and $k=-\alpha$, so that $\nu_k=0$
and $\mu_k=|\alpha|>0$. Then $(W_{\pm,k})^*=e^{\pm
i\delta_k}\HH_{|\alpha|}\HH_0$ is unbounded on $L^1(r\,dr)$ by
Step~1, so $W_\pm^*$ is unbounded on $L^1(\R^2)$ and $W_\pm$ is
unbounded on $L^\infty(\R^2)$ by Step~0.

\emph{Step 4: $L^\infty$, non even flux.} Let $\alpha\notin2\Z$
and $k=0$, so that $(W_{\pm,0})^*=e^{\pm
i\delta_0}\HH_0\HH_{|\alpha|}=:T$. Assume, by contradiction,
that $T$ is bounded on $L^1(\R_+,r\,dr)$, and let $S=\mathcal
U_1T\mathcal U_1^{-1}$, which is then bounded on $L^1(\R)$ (with
the same constant, on the dense subspace $\mathcal U_1(L^1\cap
L^2)\supset C_c^\infty(\R)$, hence everywhere). Since $2\in
S_{0,|\alpha|}=\{0<\Re z<|\alpha|+2\}$,
Proposition~\ref{prop:symbol} with $p=1$ shows that
$\widehat{Sg}=\sigma\widehat g$ for $g\in C_c^\infty(\R)$, where
$\sigma(\tau)=e^{\pm i\delta_0}m_{0,|\alpha|}(2-i\tau)$ is
continuous. By Lemma~\ref{lem:sing}(i) with
$(\nu,\mu)=(0,|\alpha|)$, $\sigma(\tau)\to e^{\pm
i\delta_0}e^{\mp i\phi}$ as $\tau\to\pm\infty$ with
$\phi=-\pi|\alpha|/2$, and these two limits are different
because $\sin\phi\ne0$. This contradicts
Lemma~\ref{lem:L1-mult}. Hence $T$ is unbounded on $L^1(r\,dr)$,
and by Step~0, $W_\pm$ is unbounded on $L^\infty(\R^2)$.

Steps 3 and 4 cover every $\alpha\ne0$ (both apply when $\alpha$
is an odd integer). Finally, the corresponding failures for
$W_\pm^*$ follow from endpoint duality: boundedness of $W_\pm^*$
on $L^\infty$ would imply boundedness of $W_\pm$ on $L^1$, while
boundedness of $W_\pm^*$ on $L^1$ would imply boundedness of
$W_\pm$ on $L^\infty$.
\end{proof}

\begin{proof}[Proof of Theorem~\ref{thm:weighted}]
If $\alpha=0$, Proposition~\ref{prop:partial-wave} gives
$W_\pm=W_\pm^*=I$, hence (i).

Let $\alpha\ne0$ and put $c=(2+\beta)/p$. The inequalities in
\eqref{eq:weighted-W-range} are equivalent to
$-\rho_\alpha<c<2$, so their sufficiency follows from
Proposition~\ref{prop:weighted-suff}. We prove necessity using
two fixed angular channels.

If $c\ge2$, then $\beta\ge2p-2$. In the radial channel $k=0$ one
has $\mu_0=0$ and $\nu_0=|\alpha|>0$. Choose $h\in
C_c^\infty(0,\infty)$, $h\ge0$, $h\not\equiv0$. By
\eqref{eq:tail},
\begin{equation*}
  (\HH_{|\alpha|}\HH_0h)(r)
  =|\alpha|r^{-2}\int_0^\infty h(s)s\dd s+O(r^{-3})
\end{equation*}
as $r\to\infty$. Its $p$th power is not integrable against
$r^{\beta+1}\,dr$ when $\beta\ge2p-2$. Since $h$ is supported in
a compact annulus, the radial function
$f(r,\theta)=(2\pi)^{-1/2}h(r)$ belongs to $L^p_\beta\cap L^2$.
Proposition~\ref{prop:partial-wave} shows that $W_\pm f$ is, up
to a unimodular constant, the function in the last display.
Therefore $W_\pm$ cannot be bounded on $L^p_\beta$.

If $c\le-\rho_\alpha$, choose $k_\alpha\in\Z$ such that
$\nu_{k_\alpha}=|k_\alpha+\alpha|=\rho_\alpha$. If
$\rho_\alpha>0$, then $\nu_{k_\alpha}$ is nonintegral while
$\mu_{k_\alpha}=|k_\alpha|$ is an integer; if $\rho_\alpha=0$,
then $\alpha\in\Z\setminus\{0\}$, $k_\alpha=-\alpha$, and
$\mu_{k_\alpha}=|\alpha|>0$. In either case
\begin{equation*}
  \frac{\mu_{k_\alpha}-\nu_{k_\alpha}}2
  \notin\{0,-1,-2,\ldots\}.
\end{equation*}
For the same choice of $h$, the coefficient in
\eqref{eq:origin-asymptotic} is nonzero, since
\begin{equation*}
  b_{\rho_\alpha,\mu_{k_\alpha}}\neq0,\qquad
  \int_0^\infty h(s)s^{-\rho_\alpha-1}\dd s>0.
\end{equation*}
Hence there are $r_0,C>0$ such that
\begin{equation*}
  |\HH_{\rho_\alpha}\HH_{\mu_{k_\alpha}}h(r)|
  \ge C r^{\rho_\alpha},\qquad 0<r<r_0.
\end{equation*}
But
\begin{equation*}
  \int_0^1 r^{p\rho_\alpha}r^{\beta+1}\dd r=\infty
  \qquad\Longleftrightarrow\qquad
  \beta+2+p\rho_\alpha\le0
  \qquad\Longleftrightarrow\qquad
  c\le-\rho_\alpha.
\end{equation*}
Taking $f(r,\theta)=(2\pi)^{-1/2}e^{ik_\alpha\theta}h(r)$ and
using Proposition~\ref{prop:partial-wave}, we conclude that
$W_\pm f\notin L^p_\beta$. Thus $W_\pm$ is again unbounded. This
proves the equivalence \eqref{eq:weighted-W-range}, including
both boundary lines.

It remains to identify the adjoint range. The Banach dual of
$L^p_\beta$ under the unweighted pairing is $L^{p'}_{\beta'}$,
where
\begin{equation*}
  p'=\frac{p}{p-1},
  \qquad
  \beta'=-\frac{\beta}{p-1}.
\end{equation*}
Since the $L^2$ adjoint of $W_\pm^*$ is $W_\pm$, weighted
duality shows that $W_\pm^*$ is bounded on $L^p_\beta$ if and
only if $W_\pm$ is bounded on $L^{p'}_{\beta'}$. To spell this
out, for test functions supported away from the origin,
\begin{equation*}
  |\langle W_\pm f,g\rangle|
  =|\langle f,W_\pm^*g\rangle|
  \le \|f\|_{L^{p'}_{\beta'}}\|W_\pm^*g\|_{L^p_\beta}.
\end{equation*}
Taking the supremum over $g$ proves one implication;
interchanging $W_\pm$ and $W_\pm^*$ proves the converse, and
density identifies the extensions with the original $L^2$
operators. Applying \eqref{eq:weighted-W-range} with
$(p',\beta')$ gives
\begin{equation*}
  -2-p'\rho_\alpha<\beta'<2(p'-1)
  \qquad\Longleftrightarrow\qquad
  -2<\beta<2(p-1)+p\rho_\alpha,
\end{equation*}
which is \eqref{eq:weighted-Wstar-range}. Intersecting the two
ranges gives \eqref{eq:weighted-similarity-range}.
The strong endpoint failures are
Proposition~\ref{prop:endpoints}.
\end{proof}

\begin{remark}[Integer flux]\label{rem:gauge}
For $n\in\Z$ the Hamiltonians $H_{n+\beta}$ and $H_\beta$ are
gauge equivalent, $H_{n+\beta}=U_n^{-1}H_\beta U_n$ with
$U_nf=e^{in\theta}f$. Since $U_n$ does not commute with $H_0$,
this does not reduce $W_\pm$ for flux $n+\beta$ to $W_\pm$ for
flux $\beta$, and indeed $W_\pm\ne I$ for
$\alpha\in\Z\setminus\{0\}$. 
Thus the obstruction to endpoint boundedness is present for
all $\alpha\neq0$; the $L^1$ failure always comes from
the $|x|^{-2}$ tail in the mode $k=0$; the $L^\infty$ failure
comes from the mode $k=-\alpha$ when
$\alpha\in\Z\setminus\{0\}$, and from the jump of the symbol at
$\tau=\pm\infty$ in the mode $k=0$ when $\alpha\notin2\Z$; for
odd integer flux both mechanisms are present.
\end{remark}

\begin{remark}
We do not know whether $W_\pm$ is of weak type $(1,1)$, or
bounded from $H^1(\R^2)$ to $L^1(\R^2)$ or from $L^\infty(\R^2)$
to $\operatorname{BMO}(\R^2)$. The $|x|^{-2}$ tail responsible
for the $L^1$ failure belongs to weak $L^1(\R^2)$, and the
Hilbert transform hidden in the symbol is of weak type $(1,1)$,
so none of the obstructions found here excludes these endpoint
substitutes. The log polar change of variables used above does
not preserve weak $L^1$ norms, so the question is not settled by
the multiplier argument of Section~\ref{sec:interior}.
\end{remark}

\section{Sharp bounds for the Krein realization}
\label{sec:Krein}

We now prove Theorem~\ref{thm:krein}. The argument is short
because the Krein and Friedrichs realizations differ only in the
two critical angular channels. What is different now is the
common Mellin strip. Indeed, \eqref{eq:Krein-orders} gives
\begin{equation}\label{eq:Krein-common-strip}
  \bigcap_{k\in\Z}S_{\lambda_k,\mu_k}
  =\bigl(\eta_\alpha,2\bigr),
  \qquad \eta_\alpha=\max\{\alpha,1-\alpha\}.
\end{equation}
The left boundary comes from the more singular of the two
critical modes; the right boundary comes from the radial mode
$k=0$.

\begin{proposition}\label{prop:Krein-weighted}
Let $0<\alpha<1$, $1<p<\infty$, and $\beta\in\R$. Then
$W_\pm^{\rm K}$ is bounded on $L^p_\beta$ if and only if
\begin{equation}\label{eq:Krein-c-range}
  \eta_\alpha<\frac{2+\beta}{p}<2.
\end{equation}
\end{proposition}

\begin{proof}
Put $c=(2+\beta)/p$. Suppose first that $\eta_\alpha<c<2$. By
Proposition~\ref{prop:Krein-partial-wave},
Proposition~\ref{prop:symbol}, and the same logarithmic polar
conjugation as in \eqref{eq:conjugation}, the symbol of
$W_\pm^{\rm K}$ is
\begin{equation}\label{eq:Krein-multiplier}
  M_{c,\pm}^{\rm K}(\tau,k)
  =e^{\mp i\delta_k^{\rm K}}m_{\lambda_k,\mu_k}(c-i\tau).
\end{equation}
For $k\notin\{0,-1\}$ this is exactly the Friedrichs symbol
$M_{c,\pm}(\tau,k)$. Since $c>\eta_\alpha>0$, the Friedrichs
symbol is an $L^p(\R\times\T)$ multiplier by
Proposition~\ref{prop:weighted-suff}. Hence
\begin{equation*}
  M_{c,\pm}^{\rm K}=M_{c,\pm}
  +\sum_{j\in\{0,-1\}}\ind_{\{k=j\}}
  \bigl(M_{c,\pm}^{\rm K}(\tau,j)-M_{c,\pm}(\tau,j)\bigr)
\end{equation*}
is also an $L^p(\R\times\T)$ multiplier. To justify the last
assertion, each of the four fixed channel symbols is bounded and
satisfies the one dimensional Mikhlin estimate
\eqref{eq:fixed-mode-mikhlin}: this follows from
Proposition~\ref{prop:symbol} and the vertical Gamma and digamma
estimates, exactly as in the argument following
\eqref{eq:fixed-mode-mikhlin}. The factors $\ind_{\{k=j\}}$ are
the bounded angular projections $P_j$. The density argument
following \eqref{eq:conjugation} now proves the $L^p_\beta$
bound for $W_\pm^{\rm K}$.

We prove necessity at both boundaries. If $c\le\eta_\alpha$,
choose $k_*\in\{0,-1\}$ so that $\lambda_{k_*}=-\eta_\alpha$.
Let $h\in C_c^\infty(0,\infty)$ be nonnegative and nonzero.
Since $(\mu_{k_*}-\lambda_{k_*})/2>0$, Lemma~\ref{lem:sing}(iii)
gives
\begin{equation}\label{eq:Krein-origin-asymptotic}
  (\HH_{-\eta_\alpha}\HH_{\mu_{k_*}}h)(r)
  =b_{-\eta_\alpha,\mu_{k_*}}r^{-\eta_\alpha}
  \int_0^\infty h(s)s^{\eta_\alpha-1}\dd s
  +O(r^{1-\eta_\alpha}),
  \qquad r\to0,
\end{equation}
with a nonzero leading coefficient. Its weighted radial norm
near zero diverges precisely when
\begin{equation*}
  \beta+2-p\eta_\alpha\le0,
\end{equation*}
which is equivalent to $c\le\eta_\alpha$. The input
$(2\pi)^{-1/2}e^{ik_*\theta}h(r)$ is in $L^2\cap L^p_\beta$,
whereas its image under $W_\pm^{\rm K}$ is not in $L^p_\beta$.

If $c\ge2$, use instead the radial channel $k=0$, for which
$(\lambda_0,\mu_0)=(-\alpha,0)$. Lemma~\ref{lem:sing}(ii), now
with the negative order $\nu=-\alpha$, yields
\begin{equation}\label{eq:Krein-tail}
  (\HH_{-\alpha}\HH_0h)(r)
  =-\alpha r^{-2}\int_0^\infty h(s)s\dd s+O(r^{-3}),
  \qquad r\to\infty.
\end{equation}
This is not in $L^p((1,\infty),r^{\beta+1}\,dr)$ when
$\beta\ge2p-2$, equivalently $c\ge2$. Thus neither boundary can
be included, and \eqref{eq:Krein-c-range} is necessary.
\end{proof}

\begin{proof}[Proof of Theorem~\ref{thm:krein}]
Existence, unitarity, and the channel formula are contained in
Proposition~\ref{prop:Krein-partial-wave}. Since
\eqref{eq:Krein-c-range} is equivalent to
$p\eta_\alpha-2<\beta<2(p-1)$,
Proposition~\ref{prop:Krein-weighted} proves
\eqref{eq:Krein-W-range}. Weighted duality, in exactly the form
used in the proof of Theorem~\ref{thm:weighted}, says that
$(W_\pm^{\rm K})^*$ is bounded on $L^p_\beta$ if and only if
$W_\pm^{\rm K}$ is bounded on $L^{p'}_{-\beta/(p-1)}$. Applying
\eqref{eq:Krein-W-range} with these dual exponents gives
\begin{equation*}
  -2<\beta<2(p-1)-p\eta_\alpha,
\end{equation*}
which is \eqref{eq:Krein-Wstar-range}.

For $\beta=0$, the wave operator condition is $p<2/\eta_\alpha$
and the adjoint condition is $p>2/(2-\eta_\alpha)$, proving
\eqref{eq:Krein-unweighted}. In their intersection, the $L^2$
identities $(W_\pm^{\rm K})^*W_\pm^{\rm K}=W_\pm^{\rm
K}(W_\pm^{\rm K})^*=I$ extend by density to $L^p$. Intersecting
the two weighted ranges gives \eqref{eq:Krein-similarity-range},
and the same density argument extends the inverse identities to
  these weighted spaces. The intertwining identity then proves
  the stated similarity on these weighted spaces.
\end{proof}

\section{Spectral similarity and transference}
\label{sec:applications}

We collect consequences of the wave operator bounds.
Let $H$ be either $\mathcal L_{\mathbf A,a}$ under the
assumptions of Theorem~\ref{thm:general}, an AB Friedrichs
Hamiltonian $H_\alpha$ with arbitrary real flux, or an AB
Krein Hamiltonian $H_\alpha^{\rm K}$ with $0<\alpha<1$.
Set $W=W_+(H,-\Delta)$ or $W=W_-(H,-\Delta)$, with one sign
fixed. The interval of simultaneous unweighted boundedness
used below is
\begin{equation*}
  I_H=
  \begin{cases}
    (1,\infty),&\text{in the Friedrichs cases},\\
    (2/(2-\eta_\alpha),\,2/\eta_\alpha),
      &H=H_\alpha^{\rm K}.
  \end{cases}
\end{equation*}

\subsection{Similarity and spectral estimates}

\begin{corollary}[$L^p$ similarity and transference]
\label{cor:transfer}
  Let $H,W$ be as above and let $p,q\in I_H$. If $m$ is a
  bounded Borel function on $[0,\infty)$ and $m(-\Delta)$
  extends boundedly from $L^p$ to $L^q$, then so does $m(H)$,
  and
  \begin{equation}\label{eq:general-transfer}
    \|m(H)\|_{L^p\to L^q}
    \le \|W\|_{L^q\to L^q}\,
    \|m(-\Delta)\|_{L^p\to L^q}\,
    \|W^*\|_{L^p\to L^p}.
  \end{equation}
  For $p=q$, this gives the similarity identity
  \begin{equation*}
    m(H)=Wm(-\Delta)W^{-1}
    \qquad\text{on }L^p.
  \end{equation*}
  For $H_\alpha$ the same identity holds on $L^p_\beta$ when
  $-2<\beta<2(p-1)$, and for $H_\alpha^{\rm K}$ when
  $p\eta_\alpha-2<\beta<2(p-1)-p\eta_\alpha$, provided the
  free multiplier is bounded on that space. The estimates
  are uniform for any uniformly bounded family of free
  multipliers.
\end{corollary}

\begin{proof}
  Unitarity and intertwining give
  $m(H)=Wm(-\Delta)W^*$ on $L^2$. For $p,q\in I_H$, the
  two wave operator factors are bounded on the required
  spaces by Theorem~\ref{thm:general},
  Corollary~\ref{cor:AB-unweighted}, or
  Theorem~\ref{thm:krein}. Applying the identity on
  $L^2\cap L^p$ proves \eqref{eq:general-transfer}, and
  density gives the extension. The identities
  $W^*W=WW^*=I$ also extend by density, proving similarity.
  The weighted assertions follow in the same way from
  Theorems~\ref{thm:weighted} and \ref{thm:krein}.
\end{proof}

The same argument acts pointwise in time. A free estimate
\begin{equation}\label{eq:spacetime-transfer}
  \|m_t(-\Delta)g\|_{L_t^rL_x^q}\le C\|g\|_{L^p},
  \qquad p,q\in I_H,
\end{equation}
transfers to $m_t(H)$ with the two wave operator norms in
\eqref{eq:general-transfer}. For Sobolev data, the spectral
intertwining also gives, on the natural domains,
\begin{equation*}
  \|(-\Delta)^{s/2}W^*h\|_2=\|H^{s/2}h\|_2,\qquad
  \|(1-\Delta)^{s/2}W^*h\|_2=\|(1+H)^{s/2}h\|_2.
\end{equation*}
This transfers the corresponding free Schr\"odinger, wave,
and Klein--Gordon Strichartz estimates whenever the required
spatial exponents lie in $I_H$. For the AB Friedrichs model,
this provides alternative proofs of estimates obtained
directly in \cite{FFFP2013,FZZ2022,GYZZ2022},
although it does not recover their endpoint $L^1\to L^\infty$ 
dispersive estimates. We give three concise examples.

\begin{enumerate}
  \item[(a)] \emph{Uniform resolvents.} A free bound
    \begin{equation*}
      \|(-\Delta-z)^{-1}\|_{L^p\to L^q}\le C_{p,q}(z),
      \qquad z\in\C\setminus[0,\infty),
    \end{equation*}
    gives the corresponding bound for $(H-z)^{-1}$, with the
    wave operator factors in \eqref{eq:general-transfer},
    whenever $p,q\in I_H$. Uniformity in $z$ is preserved.
    In the AB Friedrichs case this recovers the interior
    part of \cite{FZZ2023}; electromagnetic resolvent
    transference also appears in \cite{FSWZZ2026}.
    Cases with $p=1$ or $q=\infty$ do not follow.

  \item[(b)] \emph{Fractional integration.} Let $0<s<2$,
    $1<p<2/s$, $1/q=1/p-s/2$, and $p,q\in I_H$. Then
    \begin{equation*}
      \|H^{-s/2}f\|_{L^q}\le C_{p,s,H}\|f\|_{L^p}.
    \end{equation*}
    To justify the unbounded multiplier, use
    \begin{equation*}
      m_{\varepsilon,R}(\lambda)
      =\frac1{\Gamma(s/2)}\int_\varepsilon^R
      t^{s/2-1}e^{-t\lambda}\dd t.
    \end{equation*}
    Its free kernel is dominated uniformly by the Riesz
    kernel $C_s|x|^{s-2}$. The
    Hardy--Littlewood--Sobolev inequality and
    \eqref{eq:general-transfer} therefore give uniform
    $L^p\to L^q$ bounds. Equivalently, the extension is
    $W(-\Delta)^{-s/2}W^*$; bounded truncations and the
    spectral theorem identify it with $H^{-s/2}$ on its
    natural $L^2$ domain. For AB Friedrichs Hamiltonians,
    the same estimate follows from the diamagnetic
    inequality and the classical fractional integration
    theorem \cite{SteinWeiss1971}.

  \item[(c)] \emph{Square functions.} Let
    $\psi\in C_c^\infty((0,\infty))$ satisfy
    $\sum_{j\in\Z}|\psi(2^{-j}\lambda)|^2=1$ for $\lambda>0$.
    For every $p\in I_H$,
    \begin{equation*}
      \left\|\left(\sum_{j\in\Z}
      |\psi(2^{-j}H)f|^2\right)^{1/2}\right\|_{L^p}
      \simeq_{p,H}\|f\|_{L^p}.
    \end{equation*}
    The Marcinkiewicz--Zygmund vector valued extension of
    the bounds for $W,W^*$ transfers the classical
    Littlewood--Paley theorem. In the AB Friedrichs case,
    this also follows from the diamagnetic heat kernel
    bound and standard spectral multiplier theory.
\end{enumerate}

Other free spectral multiplier and Bochner--Riesz estimates
transfer in the same way for exponents in $I_H$. No strong
endpoint estimate is automatic, in agreement with
Theorem~\ref{thm:weighted}.

\subsection{Intertwining distinct magnetic fluxes}

The AB weighted bounds also give spectral intertwiners
between different fluxes, with no electric potential.

\begin{corollary}[Changing the flux]\label{cor:change-flux}
Let $0<a,b<1/2$ and, with the Hankel transforms of
\eqref{eq:hankel}, set
\begin{equation}\label{eq:change-flux}
  T_{a,b}=\bigoplus_{k\in\Z}\HH_{|k+a|}\HH_{|k+b|}
  \qquad\text{on }L^2(\R^2).
\end{equation}
Then $T_{a,b}$ and $T_{b,a}=T_{a,b}^{-1}$ extend boundedly to
$L^p_\beta$ for $1<p<\infty$ and $-2<\beta<2(p-1)$, in
particular to every unweighted $L^p$ with $1<p<\infty$. For
every bounded Borel function $m$ on $[0,\infty)$,
\begin{equation*}
  m(H_a)T_{a,b}=T_{a,b}m(H_b)\qquad\text{on }L^2.
\end{equation*}
This identity extends to $L^p_\beta$ whenever $m(H_b)$ is
bounded there.
\end{corollary}

\begin{proof}
Write $W_{+,a}=W_+(H_a,H_0)$ and let $P_+$, $P_-$ denote the
angular projections onto $k\ge0$, $k\le-1$. The partial wave
formula \eqref{eq:partial-wave} gives
\begin{equation*}
  W_{+,a}=D_aV_a,\qquad
  D_a=e^{i\pi a/2}P_++e^{-i\pi a/2}P_-,\qquad
  V_a=\bigoplus_{k\in\Z}\HH_{|k+a|}\HH_{|k|}.
\end{equation*}
The angular Riesz projections are bounded on $L^p_\beta$, so
$D_a$ and $D_a^{-1}$ are bounded there.
Theorem~\ref{thm:weighted} therefore gives boundedness of
$V_a=D_a^{-1}W_{+,a}$ and $V_a^*=W_{+,a}^*D_a$. Since
$\HH_{|k|}^2=I$, we have $T_{a,b}=V_aV_b^*$ and
$T_{a,b}^{-1}=T_{b,a}$ on $L^2$. The spectral representations
give the intertwining identity, and density proves all
assertions on $L^p_\beta$.
\end{proof}

For $a\ne b$ in this interval, the increasing angular
eigenvalues have the same mode ordering $0,-1,1,-2,2,\ldots$,
yet
\begin{equation*}
  |k+a|^2-|k+b|^2=2(a-b)k+a^2-b^2
\end{equation*}
is unbounded. Thus bounded spectral intertwining is compatible
with unbounded separation of the paired angular eigenvalues.
This complements the fixed magnetic potential framework of
\cite{FSWZZ2026} and the discussion at the beginning of its
Section~2. The operators $T_{a,b}$ are unphased spectral
intertwiners; the phases of the time dependent wave operators
must be accounted for separately.

\appendix
\section{A self contained transplantation argument}
\label{app:transplantation}

We prove Proposition~\ref{prop:em-transplantation} without
using the transplantation theorem of \cite{FSWZZ2026}.
The angular change of basis is already controlled by
Lemmas~\ref{lem:em-angular} and \ref{lem:em-eigenbasis},
and the remaining issue is to sum the radial Hankel products
uniformly over the angular modes. The lemma below isolates
this issue. 

\subsection{Diagonal Hankel transplantation}

\begin{lemma}\label{lem:diagonal-transplantation}
  Let $\alpha\in\R$, set $\mu_k=|k+\alpha|$, and let
  $\nu_k\ge0$, $k\in\Z$, satisfy
  \begin{equation}\label{eq:app-order-shifts}
    C_0\coloneq\sup_{k\in\Z}
    (1+|k|)|\nu_k-\mu_k|<\infty.
  \end{equation} 
  Define on $L^2(\R^2)$
  \begin{equation}\label{eq:app-diagonal}
    T\Bigl(\sum_k f_ke_k\Bigr)
    =\sum_k(\HH_{\nu_k}\HH_{\mu_k}f_k)e_k,
    \qquad e_k=(2\pi)^{-1/2}e^{ik\theta}.
  \end{equation}
  Then $T$ is unitary. For every $1<p<\infty$, both $T$
  and $T^*$ extend boundedly to $L^p(\R^2)$ and are
  mutually inverse there.
\end{lemma}

\begin{proof}
  Each Hankel transform in \eqref{eq:app-diagonal} is a
  unitary involution on $L^2(\R_+,r\,dr)$. The direct sum
  therefore defines a unitary operator, whose inverse is
  obtained by reversing the two Hankel transforms in each
  channel. We now prove the $L^p$ estimate. Fix
  $1<p<\infty$ and put $c=2/p\in(0,2)$.

  \emph{Step 1: the Mellin reduction.}
  Use the isometry $\mathcal U_{p,0}$ in
  \eqref{eq:fulliso}, and the dense space $\mathcal D$
  of finite angular sums with compactly supported smooth
  radial coefficients from Section~\ref{sec:interior}.
  Since $\mu_k,\nu_k\ge0$, the line $\Re z=c$ lies in
  every strip $S_{\nu_k,\mu_k}$. The Mellin identity of
  Proposition~\ref{prop:symbol} gives, for $f\in\mathcal D$,
  \begin{equation}\label{eq:app-conjugation}
    \mathcal U_{p,0}Tf=T_{M_c}\mathcal U_{p,0}f,
    \qquad M_c(\xi,k)=m_{\nu_k,\mu_k}(c-i\xi).
  \end{equation}
  Here $\xi$ is the Fourier variable dual to the logarithmic
  radius. The cylinder multiplier $T_{M_c}$ is initially
  understood on finite angular sums.
  We will prove its $L^p$ boundedness by constructing smooth
  extensions of the two large angular tails.

  \emph{Step 2: interpolation of the order shifts.}
  Choose $\chi\in C_c^\infty((-1/3,1/3))$ with
  $0\le\chi\le1$ and $\chi(0)=1$. For a large integer
  $K>|\alpha|+1$ and $\epsilon\in\{-1,1\}$, set
  \begin{equation}\label{eq:app-interpolation}
    h_\epsilon(x)=\frac12\sum_{n\ge K}
    (\nu_{\epsilon n}-\mu_{\epsilon n})\chi(x-n),
    \qquad x\ge K-1.
  \end{equation}
  The supports in this sum are disjoint. Thus the sum is
  locally finite and smooth, and at each point at most one
  summand or its derivative is nonzero. On its support,
  $|x-n|<1/3$. Condition~\eqref{eq:app-order-shifts} gives
  \begin{equation}\label{eq:app-interpolation-bounds}
    |h_\epsilon(x)|+|h_\epsilon'(x)|
    \le\frac{C}{1+x},\qquad
    h_\epsilon(n)=\frac{\nu_{\epsilon n}
      -\mu_{\epsilon n}}2\quad(n\ge K),
  \end{equation}
  with $C$ independent of $K$. Also, $h_\epsilon$ vanishes
  on $[K-1,K-1/3]$. Increase $K$ so that, for $x\ge K-1$,
  both $x+\epsilon\alpha\ge6$ and
  $|h_\epsilon(x)|\le1$ hold for the two signs.

  With $G_s(\zeta)=\Gamma(\zeta+s)/\Gamma(\zeta)$ as in
  \eqref{eq:Ga}, define
  \begin{equation}\label{eq:app-tail-symbol}
    \begin{gathered}
      w_\epsilon=\frac{x+\epsilon\alpha+c-i\xi}{2},\qquad
      v_\epsilon=\frac{x+\epsilon\alpha+2-c+i\xi}{2},\\
      F_\epsilon(\xi,x)
      =\frac{G_{h_\epsilon(x)}(w_\epsilon)}
        {G_{h_\epsilon(x)}(v_\epsilon)}.
    \end{gathered}
  \end{equation}
  All Gamma arguments have real part at least $2$.
  At $x=n\ge K$, one has
  $\mu_{\epsilon n}=n+\epsilon\alpha$; hence the
  interpolation and \eqref{eq:symbol-G} give the exact
  identities
  \begin{equation}\label{eq:app-tail-identity}
    F_\epsilon(\xi,n)=M_c(\xi,\epsilon n),
    \qquad n\ge K.
  \end{equation}
  Our next goal is to prove that $F_{\epsilon}$ is a multiplier
  satisfying \eqref{eq:marcinkiewicz}.

  \emph{Step 3: the variable order estimates.}
  Fix one sign and suppress $\epsilon$ in $h,w,v,F$.
  Put $R=1+x+|\xi|$. Uniformly for $|s|\le1$,
  the four quantities $|w|$, $|v|$, $|w+s|$, $|v+s|$
  are comparable to $R$. Lemma~\ref{lem:gamma}, with
  shift parameter in $[-1,1]$, therefore gives
  \begin{equation*}
    |F|\le C(|w|/|v|)^{h(x)}\le C.
  \end{equation*}
  For the derivatives, write
  \begin{equation*}
    A_s(\zeta)=\psi(\zeta+s)-\psi(\zeta),\qquad
    B_s(\zeta)=\psi'(\zeta+s)-\psi'(\zeta).
  \end{equation*}
  The same lemma gives, for $\zeta=w,v$ and $|s|\le1$,
  \begin{equation}\label{eq:app-digamma-differences}
    |A_s(\zeta)|\le C/R,\qquad
    |B_s(\zeta)|\le C/R^2.
  \end{equation}
  Put $D_s=\psi(w+s)-\psi(v+s)$. We also need
  \begin{equation}\label{eq:app-parameter-estimates}
    |D_s|\le C,\qquad
    |\psi'(w+s)|+|\psi'(v+s)|\le C/R.
  \end{equation}
  To verify these bounds, the uniform expansions
  \cite[Eqs.~5.11.2, 5.15.8]{NIST}, enlarged by compactness
  on the bounded part of $\Re\zeta\ge1$, give
  \begin{equation*}
    \psi(\zeta)=\log\zeta+O(|\zeta|^{-1}),\qquad
    |\psi'(\zeta)|\le C|\zeta|^{-1}.
  \end{equation*}
  The logarithm is the principal branch in that half plane.
  The real part of $\log(w+s)-\log(v+s)$ is bounded because
  the two moduli are comparable. Its imaginary part has
  absolute value less than $\pi$. This proves the first
  bound in \eqref{eq:app-parameter-estimates}; the second
  follows from the displayed estimate for $\psi'$.
  
  Differentiating $F$ directly is messy, and the computations
  are simplified by differentiating $\log F$ instead.
  To avoid picking an explicit logarithm of $F$ (which of course
  can be done), define instead
  \begin{equation*}
    L(\xi,x)=\int_0^{h(x)}D_s(\xi,x)\,ds.
  \end{equation*}
  The logarithmic derivative of $G_s(w)/G_s(v)$ with
  respect to $s$ is $D_s$, and this quotient equals $1$
  at $s=0$. Consequently $F=e^L$. Differentiating the
  integral, including its variable upper limit, gives
  \begin{equation}\label{eq:app-log-derivatives}
    \begin{aligned}
      L_\xi
      &=-\frac i2\bigl(A_h(w)+A_h(v)\bigr),\\
      L_x
      &=\frac12\bigl(A_h(w)-A_h(v)\bigr)+h'D_h,\\
      L_{\xi x}
      &=-\frac i4\bigl(B_h(w)+B_h(v)\bigr)
        -\frac{ih'}2
        \bigl(\psi'(w+h)+\psi'(v+h)\bigr).
    \end{aligned}
  \end{equation}
  Indeed, $w_\xi=-i/2$, $v_\xi=i/2$, and $w_x=v_x=1/2$.
  Integrating $\psi'$ from $0$ to $h$ produces the
  differences $A_h$ in the first two lines. Differentiating
  the first line with respect to $x$ produces $B_h/2$
  from each argument and $h'\psi'$ from each shift, which
  yields the last line.

  Since $|h'|\le C/x$ by
  \eqref{eq:app-interpolation-bounds} and $R\ge x$,
  equations~\eqref{eq:app-digamma-differences}--
  \eqref{eq:app-log-derivatives} imply
  \begin{equation*}
    |L_\xi|\le C/R,\qquad |L_x|\le C/x,\qquad
    |L_{\xi x}|\le C/R^2+C/(xR)\le C/(xR).
  \end{equation*}
  Now $F_\xi=FL_\xi$, $F_x=FL_x$, and
  $F_{\xi x}=F(L_{\xi x}+L_\xi L_x)$. Combining these
  identities with the bound for $F$ proves
  \begin{equation}\label{eq:app-mixed-bounds}
    |F|\le C,\qquad |F_\xi|\le C/R,\qquad
    |F_x|\le C/x,\qquad |F_{\xi x}|\le C/(xR).
  \end{equation}
  In particular, multiplication by $|\xi|$, $x$, or
  $|\xi|x$ in the respective derivative bounds gives
  exactly the product estimates
  \eqref{eq:marcinkiewicz}. Notice that the $x$ derivative
  need not decay as $|\xi|\to\infty$: the term $h'D_h$
  is the additional term absent from the constant shift
  calculation in Lemma~\ref{lem:mixed}.

  \emph{Step 4: restriction and the finite modes.}
  Extend the two tails to $\R^2$ by
  \begin{equation}\label{eq:app-extension}
    Q_\epsilon(\xi,\eta)=
    \begin{cases}
      F_\epsilon(\xi,\epsilon\eta),
        &\epsilon\eta>K-1,\\
      1,&\epsilon\eta\le K-1.
    \end{cases}
  \end{equation}
  Because $h_\epsilon=0$ on $[K-1,K-1/3]$, the first
  expression is also identically $1$ in a neighborhood of
  the junction. Thus $Q_\epsilon$ is smooth and bounded
  everywhere. Where it is not constant, $x=|\eta|$, so
  \eqref{eq:app-mixed-bounds} verifies every estimate in
  \eqref{eq:marcinkiewicz}. Theorem~\ref{thm:marcinkiewicz}
  makes $Q_\epsilon$ an $L^p(\R^2)$ multiplier, and
  Theorem~\ref{thm:deleeuw} makes its restriction to
  $\R\times\Z$ an $L^p$ multiplier on the cylinder.

  At each integer $|k|<K$ both extensions equal $1$.
  On either tail, one extension equals $M_c$ by
  \eqref{eq:app-tail-identity} and the other equals $1$.
  Therefore the following decomposition is exact:
  \begin{equation}\label{eq:app-decomposition}
    M_c(\xi,k)
      =Q_+(\xi,k)+Q_-(\xi,k)-1 +\sum_{|j|<K}\ind_{\{k=j\}}
        \bigl(m_{\nu_j,\mu_j}(c-i\xi)-1\bigr).
  \end{equation}
  For each fixed $j$, Proposition~\ref{prop:symbol} and
  the argument for \eqref{eq:fixed-mode-mikhlin} give
  \begin{equation*}
    \sup_{\xi\in\R}\left(
      |m_{\nu_j,\mu_j}(c-i\xi)|
      +(1+|\xi|)
      |\partial_\xi m_{\nu_j,\mu_j}(c-i\xi)|
    \right)<\infty.
  \end{equation*}
  That argument applies to any fixed nonnegative orders:
  the Gamma arguments have positive real part; on bounded
  $\xi$ intervals the symbol and its derivative are smooth;
  and the vertical Gamma and digamma expansions give the
  stated bounds for large $|\xi|$. In particular, zero
  values of $\mu_j$ or $\nu_j$ cause no difficulty because
  $0<c<2$. The one dimensional Mikhlin theorem therefore
  controls the radial multiplier in every summand of
  \eqref{eq:app-decomposition}. The factor
  $\ind_{\{k=j\}}$ is the bounded angular projection $P_j$.
  Fubini's theorem bounds their composition on the cylinder.
  There are only finitely many such terms, so the whole
  symbol $M_c$ is an $L^p$ multiplier.

  \emph{Step 5: compatibility, adjoints, and inverses.}
  From \eqref{eq:app-conjugation} and the isometry we obtain
  $\|Tf\|_p\le C_p\|f\|_p$ for $f\in\mathcal D$.
  If $f\in L^2\cap L^p$, choose $f_n\in\mathcal D$
  converging to $f$ in both norms. Then $Tf_n$ converges
  in $L^p$ by the estimate and in $L^2$ by unitarity.
  A subsequence converges almost everywhere in both senses,
  so the two limits agree. This identifies the bounded
  extension with the original $L^2$ operator on the
  intersection.

  The argument works for every $1<p<\infty$, hence also
  for $p'=p/(p-1)$. For $f\in L^2\cap L^p$ and
  $g\in L^2\cap L^{p'}$, the $L^2$ adjoint identity gives
  \begin{equation*}
    |\langle T^*f,g\rangle|
    =|\langle f,Tg\rangle|
    \le C_{p'}\|f\|_p\|g\|_{p'}.
  \end{equation*}
  The resulting functional of $g$ extends to $L^{p'}$.
  By duality it is represented by some $u\in L^p$ with
  $\|u\|_p\le C_{p'}\|f\|_p$. Testing on compactly
  supported smooth functions shows that $u=T^*f$ as
  distributions, hence almost everywhere. This proves the
  bound for $T^*$ and its compatibility with the $L^2$
  adjoint. Since $T$ and $T^*$ preserve $L^2\cap L^p$,
  the identities $TT^{*}=T^{*}T=I$ on $L^2$ hold on that
  intersection, hence on $L^p$ by density.
\end{proof}

\subsection{The angular transplantation}

\begin{proof}[Proof of
  Proposition~\ref{prop:em-transplantation}]
  Use the eigenbasis of Lemma~\ref{lem:em-eigenbasis}.
  The angular nonnegativity assumption gives $\tau_k\ge0$.
  For sufficiently large $|k|$, its eigenvalue estimate
  implies
  \begin{equation*}
    |\tau_k-\sigma_k|
    =\frac{|\lambda_k-(k+\alpha)^2|}
      {\tau_k+\sigma_k}
    \le\frac{\|a\|_\infty}{|k+\alpha|}
    \le\frac{C_{\alpha,a}}{1+|k|}.
  \end{equation*}
  All remaining orders are finite, so increasing the
  constant gives \eqref{eq:app-order-shifts} for
  $\mu_k=\sigma_k$ and $\nu_k=\tau_k$. Let $T$ be the
  corresponding operator in \eqref{eq:app-diagonal}.
  Lemma~\ref{lem:diagonal-transplantation} proves the
  $L^p$ bounds for $T,T^*$. Lemmas~\ref{lem:em-angular}
  and \ref{lem:em-eigenbasis} prove those for $J,J^*$.

  Applying the operators to finite angular sums gives
  $S=JT$ and $S^*=T^*J^*$ on $L^2$; density extends these
  identities to all of $L^2$. Thus $S$ is unitary and both
  products have bounded $L^p$ extensions. These extensions
  agree with the $L^2$ products on $L^2\cap L^p$, since
  each factor has that compatibility. The inverse
  identities extend to $L^p$ by density as in Step~5 above.

  Finally, the scalar spectral representation
  $h_s=\HH_s r^2\HH_s$ gives, for bounded Borel $m$,
  \begin{equation*}
    m(h_{\tau_k})\HH_{\tau_k}\HH_{\sigma_k}
    =\HH_{\tau_k}m(r^2)\HH_{\sigma_k}
    =\HH_{\tau_k}\HH_{\sigma_k}m(h_{\sigma_k}).
  \end{equation*}
  In the respective angular eigenbases, $H_{\alpha,a}$
  and $H_\alpha$ are the direct sums of $h_{\tau_k}$
  and $h_{\sigma_k}$. The displayed identity therefore
  proves $m(H_{\alpha,a})S=Sm(H_\alpha)$ on finite angular
  sums. All factors are bounded on $L^2$, so density
  proves the identity on the whole space.
\end{proof}

\begingroup
\raggedbottom

\endgroup


\begin{thebibliography}{99}

\bibitem{AdamiTeta1998} R.~Adami and A.~Teta, \emph{On the
Aharonov--Bohm Hamiltonian}, Lett. Math. Phys. \textbf{43}
(1998), 43--53.

\bibitem{ArtbazarYajima2000} G.~Artbazar and K.~Yajima,
\emph{The $L^p$ continuity of wave operators for one
dimensional Schr\"odinger operators}, J. Math. Sci. Univ.
Tokyo \textbf{7} (2000), 221--240.

\bibitem{Beceanu2014} M.~Beceanu, \emph{Structure of wave
operators for a scaling critical class of potentials}, Amer. J.
Math. \textbf{136} (2014), 255--308.

\bibitem{BeceanuSchlag2020} M.~Beceanu and W.~Schlag,
\emph{Structure formulas for wave operators}, Amer. J. Math.
\textbf{142} (2020), 751--807.

\bibitem{Ciaurri2026} \'O.~Ciaurri, \emph{Uniform weighted
inequalities for the Hankel transform transplantation operator},
Rev. R. Acad. Cienc. Exactas F\'is. Nat. Ser. A Mat. RACSAM
\textbf{120} (2026), Paper No.~5.
\url{https://doi.org/10.1007/s13398-025-01799-w}.

\bibitem{CMY2019} H.~D. Cornean, A.~Michelangeli, and K.~Yajima,
\emph{Two dimensional Schr\"odinger operators with point
interactions: threshold expansions, zero modes and $L^p$
boundedness of wave operators}, Rev. Math. Phys. \textbf{31}
(2019), 1950012.

\bibitem{DAnconaFanelli2006} P.~D'Ancona and L.~Fanelli,
\emph{$L^p$ boundedness of the wave operator for the one
dimensional Schr\"odinger operator}, Comm. Math. Phys.
\textbf{268} (2006), 415--438.
\url{https://doi.org/10.1007/s00220-006-0098-x}.

\bibitem{deLeeuw1965} K.~de Leeuw, \emph{On $L_p$ multipliers},
Ann.\ of Math. (2) \textbf{81} (1965), 364--379.

\bibitem{DerezinskiRichard2017} J.~Derezi\'nski and S.~Richard,
\emph{On Schr\"odinger operators with inverse square potentials
on the half line}, Ann. Henri Poincar\'e \textbf{18} (2017),
869--928.

\bibitem{DMSY2018} G.~Dell'Antonio, A.~Michelangeli,
R.~Scandone, and K.~Yajima, \emph{$L^p$ boundedness of wave
operators for the three dimensional multi centre point
interaction}, Ann. Henri Poincar\'e \textbf{19} (2018),
283--322.

\bibitem{FFFP2013} L.~Fanelli, V.~Felli, M.~A. Fontelos, and
A.~Primo, \emph{Time decay of scaling critical electromagnetic
Schr\"odinger flows}, Comm. Math. Phys. \textbf{324} (2013),
1033--1067.

\bibitem{FSWZZ2026} L.~Fanelli, X.~Su, Y.~Wang, J.~Zhang, and
J.~Zheng, \emph{Intertwining operators beyond the Stark effect},
Comm. Math. Phys. \textbf{407} (2026), Paper No.~92.
\url{https://doi.org/10.1007/s00220-026-05600-w}.

\bibitem{FZZ2022} L.~Fanelli, J.~Zhang, and J.~Zheng,
\emph{Dispersive estimates for 2D wave equations with critical
potentials}, Adv. Math. \textbf{400} (2022), 108333.

\bibitem{FZZ2023} L.~Fanelli, J.~Zhang, and J.~Zheng,
\emph{Uniform resolvent estimates for critical magnetic
Schr\"odinger operators in 2D}, Int. Math. Res. Not. IMRN
(2023), no.~20, 17656--17703.

\bibitem{FincoYajima2006} D.~Finco and K.~Yajima, \emph{The
$L^p$ boundedness of wave operators for Schr\"odinger operators
with threshold singularities. II\@. Even dimensional case}, J.
Math. Sci. Univ. Tokyo \textbf{13} (2006), 277--346.

\bibitem{Fermi2024} D.~Fermi, \emph{The Aharonov--Bohm
Hamiltonian: selfadjointness, spectral and scattering
properties}, arXiv:2407.15115 (2024).

\bibitem{GYZZ2022} X.~Gao, J.~Yin, J.~Zhang, and J.~Zheng,
\emph{Decay and Strichartz estimates in critical electromagnetic
fields}, J. Funct. Anal. \textbf{282} (2022), 109350.

\bibitem{Grafakos2014} L.~Grafakos, \emph{Classical Fourier
Analysis}, 3rd ed., Graduate Texts in Mathematics \textbf{249},
Springer, New York, 2014.

\bibitem{JensenYajima2002} A.~Jensen and K.~Yajima, \emph{A
remark on $L^p$ boundedness of wave operators for two
dimensional Schr\"odinger operators}, Comm. Math. Phys.
\textbf{225} (2002), 633--637.

\bibitem{MSZ2023} C.~Miao, X.~Su, and J.~Zheng, \emph{The
$W^{s,p}$ boundedness of stationary wave operators for the
Schr\"odinger operator with inverse square potential}, Trans.
Amer. Math. Soc. \textbf{376} (2023), 1739--1797.
\url{https://doi.org/10.1090/tran/8823}.

\bibitem{NIST} F.~W.~J. Olver et al. (eds.), \emph{NIST Digital
Library of Mathematical Functions},
\url{https://dlmf.nist.gov/}, Release 1.2.7 of 2026-06-15,
accessed 2 September 2026.

\bibitem{NowakStempak2006} A.~Nowak and K.~Stempak,
\emph{Weighted estimates for the Hankel transform
transplantation operator}, Tohoku Math. J. (2) \textbf{58}
(2006), 277--301. \url{https://doi.org/10.2748/tmj/1156256405}.

\bibitem{PankrashkinRichard2011} K.~Pankrashkin and S.~Richard,
\emph{Spectral and scattering theory for the Aharonov--Bohm
operators}, Rev. Math. Phys. \textbf{23} (2011), 53--81.

\bibitem{Richard2009} S.~Richard, \emph{New formulae for the
Aharonov--Bohm wave operators}, in: \emph{Spectral and
Scattering Theory for Quantum Magnetic Systems}, Contemp. Math.
\textbf{500}, Amer. Math. Soc., Providence, RI, 2009,
pp.~159--168.

\bibitem{Ruijsenaars1983} S.~N.~M. Ruijsenaars, \emph{The
Aharonov--Bohm effect and scattering theory}, Ann. Physics
\textbf{146} (1983), 1--34.

\bibitem{Saeki1970} S.~Saeki, \emph{Translation invariant
operators on groups}, T\^ohoku Math. J. (2) \textbf{22} (1970),
409--419.

\bibitem{SteinWeiss1971} E.~M. Stein and G.~Weiss,
\emph{Introduction to Fourier Analysis on Euclidean Spaces},
Princeton University Press, Princeton, NJ, 1971.

\bibitem{Stempak2002} K.~Stempak, \emph{On connections between
Hankel, Laguerre and Jacobi transplantations}, Tohoku Math. J.
(2) \textbf{54} (2002), 471--493.
\url{https://doi.org/10.2748/tmj/1113247646}.

\bibitem{Titchmarsh1948} E.~C. Titchmarsh, \emph{Introduction to
the Theory of Fourier Integrals}, 2nd ed., Clarendon Press,
Oxford, 1948.

\bibitem{Watson1944} G.~N. Watson, \emph{A Treatise on the
Theory of Bessel Functions}, 2nd ed., Cambridge University
Press, Cambridge, 1944.

\bibitem{Weder1999} R.~Weder, \emph{The $W_{k,p}$ continuity of
the Schr\"odinger wave operators on the line}, Comm. Math. Phys.
\textbf{208} (1999), 507--520.

\bibitem{Yajima1995} K.~Yajima, \emph{The $W^{k,p}$ continuity
of wave operators for Schr\"odinger operators}, J. Math. Soc.
Japan \textbf{47} (1995), 551--581.

\bibitem{Yajima1999} K.~Yajima, \emph{$L^p$ boundedness of wave
operators for two dimensional Schr\"odinger operators}, Comm.
Math. Phys. \textbf{208} (1999), 125--152.

\bibitem{Yajima2006} K.~Yajima, \emph{The $L^p$ boundedness of
wave operators for Schr\"odinger operators with threshold
singularities. I. The odd dimensional case}, J. Math. Sci. Univ.
Tokyo \textbf{13} (2006), 43--93.

\end{thebibliography}
\end{document}